\documentclass[a4paper,11pt]{article}
\usepackage{aligned-overset}
\usepackage{amsmath,amsthm}
\usepackage{authblk}
\usepackage[style=numeric-comp,maxbibnames=99,bibencoding=utf8,firstinits=true]{biblatex}
\usepackage{booktabs}
\usepackage{cancel}
\usepackage{cases}
\usepackage{colortbl}
\usepackage[hmargin={26mm,26mm},vmargin={30mm,35mm}]{geometry}
\usepackage{nicefrac}
\usepackage[colorlinks,allcolors={blue}]{hyperref}
\usepackage{paralist}
\usepackage{subcaption}
\usepackage{enumitem}
\usepackage[compact,small]{titlesec}
\usepackage{tikz-cd}

\usepackage{newtxtext}
\usepackage{newtxmath}

\bibliography{ns-ddr}

\newcommand{\email}[1]{\href{mailto:#1}{#1}}

\numberwithin{equation}{section}

\newtheorem{theorem}{Theorem}
\newtheorem{proposition}[theorem]{Proposition}
\newtheorem{lemma}[theorem]{Lemma}

\theoremstyle{remark}
\newtheorem{remark}[theorem]{Remark}
\theoremstyle{definition}

\newcommand{\st}{\,:\,}
\newcommand{\Real}{\mathbb{R}}
\newcommand{\Natural}{\mathbb{N}}

\DeclareRobustCommand{\bvec}[1]{\boldsymbol{#1}}
\newcommand{\uvec}[1]{\underline{\bvec{#1}}}
\newcommand{\cvec}[1]{\bvec{\mathcal{#1}}}

\DeclareMathOperator{\GRAD}{\bf grad}
\DeclareMathOperator{\CURL}{\bf curl}
\DeclareMathOperator{\DIV}{div}

\DeclareMathOperator{\VROT}{\bf rot}

\newcommand{\compl}{{\rm c}}
 
\newcommand{\Hcurl}[1]{\bvec{H}(\CURL;#1)}
\newcommand{\Hcurlz}[1]{\bvec{H}_0(\CURL;#1)}

\newcommand{\Hdiv}[1]{\bvec{H}(\DIV;#1)}

\newcommand{\Xgrad}[2]{\underline{X}_{\GRAD,#2}^{#1}}
\newcommand{\Xgrads}[2]{\underline{X}_{\GRAD,#2,a}^{#1}}
\newcommand{\Xgradz}[2]{\underline{X}_{\GRAD,#2,0}^{#1}}
\newcommand{\Xcurl}[2]{\underline{\bvec{X}}_{\CURL,#2}^{#1}}
\newcommand{\Xcurlz}[2]{\underline{\bvec{X}}_{\CURL,#2,0}^{#1}}
\newcommand{\Xdiv}[2]{\underline{\bvec{X}}_{\DIV,#2}^{#1}}
\newcommand{\Xbullet}[2]{\underline{X}_{\bullet,#2}^{#1}}

\newcommand{\Igrad}[2]{\underline{I}_{\GRAD,#2}^{#1}}
\newcommand{\Icurl}[2]{\uvec{I}_{\CURL,#2}^{#1}}
\newcommand{\Idiv}[2]{\uvec{I}_{\DIV,#2}^{#1}}
\newcommand{\Ibullet}[2]{\underline{I}_{\bullet,#2}^{#1}}

\newcommand{\symbolproj}{\pi}
\newcommand{\proj}[2]{\symbolproj_{#2}^{#1}}
\newcommand{\lproj}[2]{\symbolproj_{\mathcal{P},#2}^{#1}}
\newcommand{\vlproj}[2]{\boldsymbol{\symbolproj}_{\cvec{P},#2}^{#1}}

\newcommand{\Rproj}[2]{\bvec{\symbolproj}_{\cvec{R},#2}^{#1}}
\newcommand{\Rcproj}[2]{\bvec{\symbolproj}_{\cvec{R},#2}^{\compl,#1}}

\newcommand{\Gproj}[2]{\bvec{\symbolproj}_{\cvec{G},#2}^{#1}}
\newcommand{\Gcproj}[2]{\bvec{\symbolproj}_{\cvec{G},#2}^{\compl,#1}}

\newcommand{\uGT}[1]{\uvec{G}_T^{#1}}

\newcommand{\uCT}[1]{\uvec{C}_T^{#1}}

\newcommand{\uGh}[1]{\uvec{G}_h^{#1}}
\newcommand{\uCh}[1]{\uvec{C}_h^{#1}}
\newcommand{\Dh}[1]{D_h^{#1}}

\newcommand{\cGF}[1]{\boldsymbol{\mathsf{G}}_F^{#1}}
\newcommand{\cGT}[1]{\boldsymbol{\mathsf{G}}_T^{#1}}

\newcommand{\CF}[1]{C_F^{#1}}
\newcommand{\cCT}[1]{\boldsymbol{\mathsf{C}}_T^{#1}} 
\newcommand{\cCh}[1]{\boldsymbol{\mathsf{C}}_h^{#1}} 

\newcommand{\DT}{D_T^k}

\newcommand{\trE}{\gamma_E^{k+1}}
\newcommand{\trF}{\gamma_F^{k+1}}
\newcommand{\trFt}{\bvec{\gamma}_{{\rm t},F}^k}

\newcommand{\faces}[1]{\mathcal{F}_{#1}}
\newcommand{\edges}[1]{\mathcal{E}_{#1}}
\newcommand{\vertices}[1]{\mathcal{V}_{#1}}

\newcommand{\FT}{\faces{T}}
\newcommand{\ET}{\edges{T}}
\newcommand{\EF}{\edges{F}}
\newcommand{\VT}{\vertices{T}}

\newcommand{\VE}{\vertices{E}}

\newcommand{\normal}{\bvec{n}}
\newcommand{\tangent}{\bvec{t}}

\newcommand{\Poly}[2][]{\mathcal{P}_{#1}^{#2}}

\newcommand{\vPoly}[2][]{\cvec{P}_{#1}^{#2}}
\newcommand{\Roly}[1]{\cvec{R}^{#1}}
\newcommand{\Goly}[1]{\cvec{G}^{#1}}
\newcommand{\cRoly}[1]{\cvec{R}^{\compl,#1}}
\newcommand{\cGoly}[1]{\cvec{G}^{\compl,#1}}
\newcommand{\Koly}[2]{\mathcal{K}^{#1}_{#2}}

\newcommand{\norm}[2]{\|#2\|_{#1}}
\newcommand{\seminorm}[2]{|#2|_{#1}}
\newcommand{\vvvert}{\vert\kern-0.25ex\vert\kern-0.25ex\vert}
\newcommand{\tnorm}[2]{\vvvert #2\vvvert_{#1}}

\DeclareMathOperator{\Image}{Im}

\newcommand{\Mh}[1][h]{\mathcal{M}_{#1}}
\newcommand{\Th}[1][h]{\mathcal{T}_{#1}}
\newcommand{\Fh}[1][h]{\mathcal{F}_{#1}}
\newcommand{\Eh}[1][h]{\mathcal{E}_{#1}}
\newcommand{\Vh}{\mathcal{V}_h}
\newcommand{\Fhb}[1][h]{\mathcal{F}_{#1}^b}
\newcommand{\Ehb}[1][h]{\mathcal{E}_{#1}^b}
\newcommand{\Vhb}{\mathcal{V}_h^b}

\newcommand{\Pgrad}[2]{P_{\GRAD,#2}^{#1}}
\newcommand{\Pcurl}[2]{\bvec{P}_{\CURL,#2}^{#1}}
\newcommand{\Pdiv}[2]{\bvec{P}_{\DIV,#2}^{#1}}
\newcommand{\Pbullet}[2]{P_{\bullet,#2}^{#1}}

\newcommand{\SG}[1]{\bvec{S}_{\GRAD,#1}^k}
\newcommand{\SC}[1]{\bvec{S}_{\CURL,#1}^k}

\newcommand{\Csob}{C_{\rm S}}
\newcommand{\Csobh}{C_{{\rm S},h}}
\newcommand{\Cpoin}{C_{{\rm p},\CURL}}
\newcommand{\CPcurl}{C_{{\rm c},\CURL}}
\newcommand{\CPdiv}{C_{{\rm c},\DIV}}

\newcommand{\ed}{{\rm d}}
\newcommand{\ded}[2]{\ed^{#1}_{#2}}
\newcommand{\ued}[2]{\underline{\ed}^{#1}_{#2}}
\newcommand{\Xsp}[2]{\underline{X}^{#1}_{#2}}
\newcommand{\Lpdf}[3]{L^{#1}\Lambda^{#2}(#3)}
\newcommand{\Wpdf}[3]{W^{#1}\Lambda^{#2}(#3)}
\newcommand{\Pkdf}[3]{\Poly{#1}\Lambda^{#2}(#3)}
\newcommand{\Pec}[2]{P^{#1}_{#2}}
\newcommand{\Iec}[2]{\underline{I}^{#1}_{#2}}
\newcommand{\tr}{\mathrm{tr}}

\usepackage{pgfplots,pgfplotstable}
\usepackage{tikz-cd}

\graphicspath{{figures/}}

\newcommand{\logLogSlopeTriangle}[5]
{
    \pgfplotsextra
    {
        \pgfkeysgetvalue{/pgfplots/xmin}{\xmin}
        \pgfkeysgetvalue{/pgfplots/xmax}{\xmax}
        \pgfkeysgetvalue{/pgfplots/ymin}{\ymin}
        \pgfkeysgetvalue{/pgfplots/ymax}{\ymax}

        \pgfmathsetmacro{\xArel}{#1}
        \pgfmathsetmacro{\yArel}{#3}
        \pgfmathsetmacro{\xBrel}{#1-#2}
        \pgfmathsetmacro{\yBrel}{\yArel}
        \pgfmathsetmacro{\xCrel}{\xArel}

        \pgfmathsetmacro{\lnxB}{\xmin*(1-(#1-#2))+\xmax*(#1-#2)} 
        \pgfmathsetmacro{\lnxA}{\xmin*(1-#1)+\xmax*#1} 
        \pgfmathsetmacro{\lnyA}{\ymin*(1-#3)+\ymax*#3} 
        \pgfmathsetmacro{\lnyC}{\lnyA+#4*(\lnxA-\lnxB)}
        \pgfmathsetmacro{\yCrel}{\lnyC-\ymin)/(\ymax-\ymin)}

        \coordinate (A) at (rel axis cs:\xArel,\yArel);
        \coordinate (B) at (rel axis cs:\xBrel,\yBrel);
        \coordinate (C) at (rel axis cs:\xCrel,\yCrel);

        \draw[#5]   (A)-- node[pos=0.5,anchor=north] {\scriptsize{1}}
                    (B)-- 
                    (C)-- node[pos=0.,anchor=west] {\scriptsize{#4}} 
                    cycle;
    }
}

\begin{document}

\title{A pressure-robust Discrete de Rham scheme for the Navier--Stokes equations}

\author[1]{Daniele A. Di Pietro}
\author[1,2]{J\'er\^ome Droniou}
\author[2]{Jia Jia Qian}
\affil[1]{IMAG, Univ Montpellier, CNRS, Montpellier, France, \email{daniele.di-pietro@umontpellier.fr}}
\affil[2]{School of Mathematics, Monash University, Melbourne, Australia, \email{jerome.droniou@umontpellier.fr} \email{jerome.droniou@monash.edu}, \email{jia.qian@monash.edu}}

\maketitle

\begin{abstract}
  In this work we design and analyse a Discrete de Rham (DDR) method for the incompressible Navier--Stokes equations.
    Our focus is, more specifically, on the SDDR variant, where a reduction in the number of unknowns is obtained using serendipity techniques.
  The main features of the DDR approach are the support of general meshes and arbitrary approximation orders.
  The method we develop is based on the curl-curl formulation of the momentum equation and, through compatibility with the Helmholtz--Hodge decomposition, delivers pressure-robust error estimates for the velocity.
  It also enables non-standard boundary conditions, such as imposing the value of the pressure on the boundary.
  In-depth numerical validation on a complete panel of tests including general polyhedral meshes is provided.
  The paper also contains an appendix where bounds on DDR potential reconstructions and differential operators are proved in the more general framework of Polytopal Exterior Calculus. \medskip
\end{abstract}



\section{Introduction}

The construction and analysis of accurate numerical approximations of the incompressible Navier--Stokes equations remains an open problem.
Classical issues in this context are related to the identification of inf-sup stable couples of velocity and pressure spaces (see, e.g.,  \cite{Girault.Raviart:86}) and to the robust handling of convection-dominated regimes.
More subtle problems, however, can arise.
It has been recently pointed out in \cite{Linke:14} that classical methods for the Navier--Stokes equations or linearised versions thereof may lack pressure-robustness, i.e., the ability to deliver accurate velocity approximations in the presence of large irrotational body forces.
This issue is tightly related to the non-compliance of these methods with the Helmholtz--Hodge decomposition of the body force term.

In the context of classical Finite Element methods on standard meshes, pressure-robustness can be obtained using $\bvec{H}(\DIV)$-conforming spaces for the velocity \cite{Falk.Neilan:13,Zhang:16} or taking projections thereon \cite{Linke.Merdon:16,Di-Pietro.Ern.ea:16*1}.
The latter strategy can be applied to polyhedral methods such as the Virtual Element or Hybrid High-Order methods through projections on $\bvec{H}(\DIV)$-conforming spaces constructed starting from a matching simplicial (sub)mesh \cite{Frerichs.Merdon:20,Castanon-Quiroz.Di-Pietro:20,Castanon-Quiroz.Di-Pietro:23}.
Working on a submesh, however, can be computationally expensive, particularly in three space dimensions, as numerical integration has to be performed in each tetrahedron.
An altogether different strategy has been recently proposed in \cite{Beirao-da-Veiga.Dassi.ea:22}, where discrete versions of the de Rham complex are used to devise compatible ($\bvec{H}(\CURL)$,$H^1$)-like space couples for the Stokes problem.
Pressure robustness in this context results from the compatibility of the $\bvec{H}(\CURL)$-like interpolate of the body force with its Helmholtz--Hodge decomposition.

Another advantage of using a formulation of the Stokes (or Navier--Stokes) problem based on the $\bvec{H}(\CURL)$ space for the velocity and the $H^1$ space for the pressure is that it allows for a seamless handling of non-standard boundary conditions \cite{Girault:90}. In the standard weak form of these equations, based on the $\bvec{H}^1$ space for the velocity and $L^2$ space for the pressure, only components of the velocity and/or the normal component of its gradient can be enforced on the boundary. However, imposing the pressure on the boundary can be quite relevant in certain applications such as, e.g., blood flows \cite{Muha.Canic:13}. To enforce such boundary conditions using ($\bvec{H}^1$,$L^2$) space couples, Lagrange multipliers must be used as in \cite{Bertoluzza.Chabannes.ea:17} for the Stokes equations, but only result in weak enforcement of boundary pressure values. On the contrary, a weak formulation based on ($\bvec{H}(\CURL)$,$H^1$) space couples enables to strongly impose the normal velocity and tangential vorticity or the tangential velocity and pressure on the boundary.

The goal of the present work is to extend the Discrete de Rham (DDR) method of \cite{Beirao-da-Veiga.Dassi.ea:22}, based on a discrete counterpart of the ($\bvec{H}(\CURL)$,$H^1$) space couple, to the full Navier--Stokes equations.
  Our focus will be, more specifically, on the SDDR variant of \cite{Di-Pietro.Droniou:23*1}, where the dimension of the discrete spaces is reduced using serendipity techniques.
The proposed scheme hinges on a naturally non-dissipative convective term designed from the discrete curl and corresponding potential.
With this choice, the aforementioned non-standard boundary conditions can be strongly enforced and, provided a uniform discrete Sobolev inequality for the curl holds, one can obtain pressure-robust and optimally convergent error estimates for the velocity as well as for a discrete $W^{1,\frac43}$-like norm of the pressure. Notice that, while the validity of Sobolev-type inequalities is a consequence of the cohomology properties of the SDDR complex, proving that the corresponding constants are independent of the meshsize is, to date, an open problem.

Instrumental to our analysis are bounds on potential reconstructions and discrete differential operators that we prove here in the more general framework of Polytopal Exterior Calculus \cite{Bonaldi.Di-Pietro.ea:23}.

The rest of this work is organised as follows.
The continuous setting is described in Section \ref{sec:continuous.setting}.
The new SDDR scheme as well as the main theoretical results are described in Section \ref{sec:scheme.results}, while the details of the analysis are postponed to Section \ref{sec:analysis}.
In Section \ref{sec:essential.BC} we briefly discuss the enforcement of essential boundary conditions.
Section \ref{sec:num.tests} contains an extensive panel of numerical tests.
Finally, the proofs of relevant bounds on potential reconstructions and discrete differential operators are provided in Appendix \ref{sec:bounds.pec}.


\section{Continuous setting}\label{sec:continuous.setting}

We consider the Navier--Stokes equations on a convex polyhedral domain $\Omega\subset\Real^3$ with trivial topology:
\begin{equation}\label{eq:strong.Delta}
  \begin{aligned}
    &\text{Find the velocity $\bvec{u}:\Omega\to\Real^3$ and the pressure $\varpi:\Omega\to\Real$ such that}
    \\
    &-\nu\Delta\bvec{u} + \DIV(\bvec{u}\otimes\bvec{u})+\GRAD \varpi = \bvec{f}\quad\text{in $\Omega$},
    \\
    &\DIV\bvec{u} = 0 \quad\text{in $\Omega$},
  \end{aligned}
\end{equation}
where $\bvec{f}:\Omega\to\Real^3$ represents the volumetric force and the real number $\nu>0$ is the viscosity of the fluid.
The pressure-robust scheme we design relies on discrete counterparts of the $\Hcurl{\Omega}$ and $H^1(\Omega)$ spaces, which are adapted to the Hodge decomposition of $\bvec{f}$. We will therefore consider the reformulation of \eqref{eq:strong.Delta} based on the curl operator, obtained applying the following identities (the second one follows from the Lamb identity):
\[
-\Delta \bvec{u}=\CURL\CURL \bvec{u} - \GRAD\DIV\bvec{u}\,,\quad \DIV(\bvec{u}\otimes\bvec{u})= (\DIV \bvec{u})\bvec{u}+\CURL \bvec{u}\times \bvec{u}+\frac12 \GRAD(\bvec{u}\cdot\bvec{u}).
\]
Denoting by $p \coloneq \varpi+\frac12 \bvec{u}\cdot\bvec{u}$ the Bernoulli pressure and taking into account the incompressibility condition $\DIV\bvec{u}= 0$ in the formulas above, the Navier--Stokes equations therefore become:
\begin{subequations}\label{eq:momentum.incompressibility}
\begin{align}
    &\text{Find the velocity $\bvec{u}:\Omega\to\Real^3$ and the pressure $p:\Omega\to\Real$ such that}
    \nonumber\\
    &\nu\CURL\CURL\bvec{u} + \CURL \bvec{u} \times \bvec{u}+\GRAD p = \bvec{f}\quad\text{in $\Omega$},
    \label{eq:momentum}\\
    &\DIV\bvec{u} = 0 \quad\text{in $\Omega$}.
    \label{eq:incompressibility}
\end{align}
\end{subequations}
Boundary conditions are needed to close the problem.
Besides providing the means to design a pressure-robust scheme for the Navier--Stokes equations, the formulation \eqref{eq:momentum.incompressibility} also allows us to consider, as in \cite{Girault:90}, non-standard boundary conditions, enforcing either the tangential vorticity and normal component of the velocity (for natural boundary conditions), or the tangential velocity and the pressure (for essential boundary conditions).
In most of this paper, we consider, for the sake of simplicity, homogeneous natural boundary conditions:
\begin{equation}\label{eq:natural.BCs}
  \CURL\bvec{u}\times\normal = \bvec{0} \quad\text{ and }\quad\bvec{u}\cdot\normal = 0 \quad\text{on $\partial\Omega$}
\end{equation}
and briefly discuss the case of essential boundary conditions in Section \ref{sec:essential.BC}.
The extension to the non-homogenous case is straightforward and is considered numerically in Section \ref{sec:num.tests}.
Natural boundary conditions lead to a model in which the pressure is only defined up to an additive constant, which we fix by imposing
\begin{equation*}
\int_\Omega p=0.
\end{equation*}

The weak formulation of this model is obtained taking the dot product of the momentum equation by a test function $\bvec{v}$, multiplying the continuity equation by a test function $q$, integrating by parts the viscous term, and fixing the Sobolev spaces for the trial and test functions to ensure that all differential quantities are well defined:
\begin{subequations}\label{eq:weak}
  \begin{align}
    &\text{Find $(\bvec{u},p)\in \Hcurl{\Omega}\times H^1(\Omega)$ such that, for all $(\bvec{v},q)\in \Hcurl{\Omega}\times H^1(\Omega)$,}\nonumber
    \\
    &\nu\int_\Omega\CURL\bvec{u}\cdot\CURL\bvec{v} + \int_\Omega [\CURL \bvec{u}\times \bvec{u}]\cdot\bvec{v}+\int_\Omega\GRAD p\cdot\bvec{v} = \int_\Omega\bvec{f}\cdot\bvec{v},
    \\ \label{eq:weak:mass}
    &-\int_\Omega\bvec{u}\cdot\GRAD q = 0,
    \\
    &\int_\Omega p = 0.
  \end{align}
\end{subequations}
In this form, the nonlinear term is naturally non-dissipative since $[\CURL \bvec{u}\times \bvec{u}]\cdot\bvec{u}=0$ by orthogonality of the cross product.
Notice that \eqref{eq:weak:mass} tested with $q \in H_0^1(\Omega)$ ensures that $\DIV\bvec{u} = 0$ so that, in particular, $\bvec{u}\in \Hdiv{\Omega}$.
  Using this fact in \eqref{eq:weak:mass} tested with a generic $q \in H^1(\Omega)$ and recalling the surjectivity of the trace operator, we additionally get $\bvec{u}\cdot\normal=0$ on $\partial\Omega$.
Combined with the fact that $\bvec{u}\in\Hcurl{\Omega}$, and since $\Omega$ is convex, the above conditions ensure, in turn, that $\bvec{u}\in \bvec{H}^1(\Omega)$ \cite{Girault:90,Amrouche.Bernardi.ea:98} and thus, by the Sobolev embedding, that the trilinear term is well-defined.


\section{Numerical scheme and main result}\label{sec:scheme.results}

\subsection{Mesh and notation}\label{sec:mesh}

Given a (measurable) set $Y\subset\Real^3$, we denote by $h_Y$ its diameter.
We consider meshes $\Mh$ defined as the union of the following sets:
$\Th$, a finite collection of open disjoint polyhedral elements such that $\overline{\Omega} = \bigcup_{T\in\Th}\overline{T}$ and $h=\max_{T\in\Th}h_T>0$;
$\Fh$, a finite collection of open planar polygonal faces;
$\Eh$, a finite collection of open straight edges;
$\Vh$, the set collecting the edge endpoints.
  We assume that $(\Th,\Fh)$ matches the conditions in \cite[Definition 1.4]{Di-Pietro.Droniou:20},
  which stipulate, in particular, that each face is contained in the boundary of some element, and that the boundary of each element is equal to the union of a subset of (closures of) faces. The same relations are assumed between edges and faces (e.g.~the boundary of each face is the union of closures of edges) and between vertices and edges. This definition is very generic and, in particular, allows for situations where a flat piece of a boundary of an element is cut into several mesh faces, which typically occurs when local mesh refinement is performed.

The set collecting the mesh faces that lie on the boundary of a mesh element $T\in\Th$ is denoted by $\FT$.
For any $Y\in\Th\cup\Fh$, we denote by $\edges{Y}$ the set of edges of $Y$.
Similarly, for all $Y\in\Th\cup\Fh\cup\Eh$, $\vertices{Y}$ denotes the set of vertices of $Y$.

For any face $F\in\Fh$, we fix a unit normal vector $\normal_F$ and, for any edge $E\in\Eh$, a unit tangent vector $\tangent_E$. For any $F \in \FT$, we let $\omega_{TF} \in \{-1, 1\}$ be such that $\omega_{TF} \normal_F$ points out of $T$.
If $F\in\Fh$ and $E\in\EF$, we denote by $\normal_{FE}$ the vector normal to $E$ in the plane containing $F$ oriented such that $(\tangent_E,\normal_{FE},\normal_F)$ forms a right-handed system of coordinates, and we set $\omega_{FE} \in \{-1, 1\}$ such that $\omega_{FE} \normal_{FE}$ points out of $F$.

For each mesh element or face $Y\in\Th\cup\Fh$, we fix a point~$\bvec{x}_Y\in Y$ such that there exists a ball centered in $\bvec{x}_Y$ contained in $Y$ and of diameter comparable to $h_Y$ uniformly in $h$ (when $\Mh$ belongs to a regular mesh sequence in the sense of \cite[Definition 1.9]{Di-Pietro.Droniou:20}).

Throughout the paper, $a\lesssim b$ stands for $a\le Cb$ with $C$ depending only on $\Omega$, the mesh regularity parameter and, when polynomial functions are involved, the corresponding polynomial degree. The notation $a\simeq b$ is a shorthand for ``$a\lesssim b$ and $b\lesssim a$''.

\subsection{Polynomial spaces}
For any integer $l\geq -1$ and any $Y\in\Mh$, we denote by $\Poly{l}(Y)$ the space of polynomial functions of total degree $\leq l$ on $Y$, with the convention that $\Poly{-1}(Y)\coloneq\left\{0\right\}$.
We use boldface to indicate vector-valued polynomial spaces.
Specifically, we set $\vPoly{l}(F)\coloneq\Poly{l}(F)^2$ for any face $F\in\Fh$ and $\vPoly{l}(T)\coloneq\Poly{l}(T)^3$ for any element $T\in\Th$.
The $L^2$-orthogonal projectors on these (full) scalar and vector-valued polynomial spaces are $\lproj{l}{Y}$ and $\vlproj{l}{Y}$ respectively.
For any face $F\in\Fh$, we define the following polynomial subspaces of $\vPoly{l}(F)$:
\begin{subequations}\label{def:GRoly.F}
  \begin{align}
    \Goly{l}(F) & \coloneq \GRAD_F\Poly{l+1}(F), &  \cGoly{l}(F) & \coloneq (\bvec{x}-\bvec{x}_F)^\perp\Poly{l-1}(F), \\
    \Roly{l}(F) & \coloneq  \VROT_F\Poly{l+1}(F), & \cRoly{l}(F) & \coloneq (\bvec{x}-\bvec{x}_F)\Poly{l-1}(F),
  \end{align}
\end{subequations}
where $\bvec{y}^\perp$ is obtained rotating the vector $\bvec{y}$ tangentially to $F$ by an angle of $-\frac{\pi}{2}$ oriented by $\normal_F$, and $\GRAD_F$ and $\VROT_F$ respectively denote the tangential gradient and rotor (rotation of the gradient $\VROT_F f=(\GRAD_F f)^\perp$).
The following direct decompositions hold:
\begin{equation*}
  \vPoly{l}(F)
  = \Goly{l}(F)\oplus\cGoly{l}(F) = \Roly{l}(F)\oplus\cRoly{l}(F).
\end{equation*}
Likewise, on any mesh element  $T\in\Th$, we define the following subspaces of $\vPoly{l}(T)$:
\begin{subequations}\label{def:GRoly.T}
  \begin{align}
    \Goly{l}(T) & \coloneq \GRAD\Poly{l+1}(T), & \cGoly{l}(T) & \coloneq (\bvec{x}-\bvec{x}_T)\times\vPoly{l-1}(T), \\
    \Roly{l}(T) & \coloneq \CURL\vPoly{l+1}(T), & \cRoly{l}(T) & \coloneq (\bvec{x}-\bvec{x}_T) \Poly{l-1}(T),
  \end{align}
\end{subequations}
which decompose the polynomial space $\vPoly{l}(T)$ as
\begin{equation*}
  \vPoly{l}(T)
  = \Goly{l}(T)\oplus\cGoly{l}(T) = \Roly{l}(T)\oplus\cRoly{l}(T).
\end{equation*}
For any $\cvec{X} \in \{ \cvec{R}, \cvec{G} \}$ and any $Y \in \Th \cup \Fh$, the $L^2$-orthogonal projectors on $\cvec{X}^l(Y)$ and $\cvec{X}^{\compl,l}(Y)$ are, respectively, $\bvec{\symbolproj}_{\cvec{X},Y}^{l}$ and $\bvec{\symbolproj}_{\cvec{X},Y}^{\compl,l}$.

\subsection{Serendipity Discrete de Rham complex}

In what follows we briefly present the main elements of the SDDR construction, namely the discrete spaces and the discrete counterparts of the differential operators gradient, curl, and divergence, and of the corresponding (scalar or vector) potentials. For a detailed presentation, we refer the reader to \cite{Di-Pietro.Droniou:23*1}.

  For the notion of continuous vector potentials for the curl and divergence operators, see, e.g., \cite{Amrouche.Bernardi.ea:98}.
  Discrete potentials can be regarded as inspired by this notion in the sense that they play the role of a potential in discrete integration by parts formulas.%

\subsubsection{SDDR spaces and serendipity operators}

Throughout the rest of the paper, we fix an integer $k\ge 0$ corresponding to the polynomial degree of the complex.
For each mesh face (resp., element) $Y$, we select a number $\eta_Y\ge 2$ of edges (resp., faces) on the boundary of $Y$ such that $Y$ lies entirely on one side of the affine hyperspace spanned by each of the selected edges (resp.~faces). We then set
\[
\ell_Y \coloneq k + 1 - \eta_Y\qquad\forall Y\in\Th\cup\Fh.
\]

The SDDR space $\Xgrad{k}{h}$, which replaces $H^1(\Omega)$ at the discrete level, and its interpolator $\Igrad{k}{h}:C^0(\overline{\Omega})\to\Xgrad{k}{h}$ are, respectively,
\begin{align*}
  \Xgrad{k}{h}
  &\coloneq
  \Big\{
  \begin{aligned}[t]
    \underline{q}_h
    &= ((q_T)_{T\in\Th},(q_F)_{F\in\Fh},(q_E)_{E\in\Eh},(q_V)_{V\in\Vh})\st
    \\
    & \text{$q_T\in\Poly{\ell_T}(T)$ for all $T\in\Th$, $q_F\in\Poly{\ell_F}(F)$ for all $F\in\Fh$},\\
    & \text{$q_E\in\Poly{k-1}(E)$ for all $E\in\Eh$, and $q_V\in\Real$ for all $V\in\Vh$}
    \Big\},
  \end{aligned}\\
  \Igrad{k}{h}q
  &\coloneq
  ((\lproj{\ell_T}{T}q)_{T\in\Th},(\lproj{\ell_F}{F}q)_{F\in\Fh},(\lproj{k-1}{E}q)_{E\in\Eh},(q(\bvec{x}_V))_{V\in\Vh})\quad\forall q\in C^0(\overline{\Omega}).
\end{align*}

The SDDR discrete $\Hcurl{\Omega}$ space and its interpolator $\Icurl{k}{h}:\bvec{C}^0(\overline{\Omega})\to \Xcurl{k}{h}$ are
\begin{align*}
  \Xcurl{k}{h}
  &\coloneq
  \Big\{
  \begin{aligned}[t]
    \uvec{v}_h
    &=((\bvec{v}_{\cvec{R},T},\bvec{v}_{\cvec{R},T}^\compl)_{T\in\Th},(\bvec{v}_{\cvec{R},F},\bvec{v}_{\cvec{R},F}^\compl)_{F\in\Fh},(v_E)_{E\in\Eh})
    \st
    \\
    & \text{$\bvec{v}_{\cvec{R},T}\in\Roly{k-1}(T)$ and $\bvec{v}_{\cvec{R},T}^\compl\in\cRoly{\ell_T+1}(T)$ for all $T\in\Th$,}\\
    & \text{$\bvec{v}_{\cvec{R},F}\in\Roly{k-1}(F)$ and $\bvec{v}_{\cvec{R},F}^\compl\in\cRoly{\ell_F+1}(F)$ for all $F\in\Fh$,}\\
    & \text{and $v_E\in\Poly{k}(E)$ for all $E\in\Eh$}      
    \Big\}, 
  \end{aligned}\\
  \Icurl{k}{h}\bvec{v}
  &\coloneq
  ( (\Rproj{k-1}{T}\bvec{v},\Rcproj{\ell_T+1}{T}\bvec{v})_{T\in\Th},(\Rproj{k-1}{F}\bvec{v}_{{\rm t},F},\Rcproj{\ell_F+1}{F}\bvec{v}_{{\rm t},F})_{F\in\Fh},(\lproj{k}{E}(\bvec{v}\cdot\tangent_E))_{E\in\Eh})\quad\forall\bvec{v}\in
  \bvec{C}^0(\overline{\Omega}),
\end{align*}
where $\bvec{v}_{{\rm t},F}$ denotes the tangential component of $\bvec{v}$ on $F$.

Finally, the discrete $\Hdiv{\Omega}$ space and its interpolator $\Idiv{k}{h}:\bvec{C}^0(\overline{\Omega})\to\Xdiv{k}{h}$ are
\begin{align*}
  \Xdiv{k}{h}
  &\coloneq
  \Big\{
  \begin{aligned}[t]
    \uvec{w}_h&=((\bvec{w}_{\cvec{G},T},\bvec{w}_{\cvec{G},T}^\compl)_{T\in\Th},(w_F)_{F\in\Fh})
    \st
    \\
    & \text{$\bvec{w}_{\cvec{G},T}\in\Goly{k-1}(T)$ and $\bvec{w}_{\cvec{G},T}^\compl\in\cGoly{k}(T)$ for all $T\in\Th$,}\\
    & \text{and $v_F\in\Poly{k}(F)$ for all $F\in\Fh$}      
    \Big\},
  \end{aligned}\\
  \Idiv{k}{h}\bvec{w}
  &\coloneq
  ( (\Gproj{k-1}{T}\bvec{w},\Gcproj{k}{T}\bvec{w})_{T\in\Th},(\lproj{k}{F}(\bvec{w}\cdot\normal_E))_{E\in\Eh})\quad\forall\bvec{w}\in \bvec{C}^0(\overline{\Omega}).
\end{align*}

The restriction of the above spaces and their element to a mesh entity $Y\in\Eh\cup\Fh\cup\Th$ are denoted by replacing the subscript $h$ by $Y$.
For example, given a mesh element $T \in \Th$, $\Xgrad{k}{T}$ denotes the restriction of $\Xgrad{k}{h}$ to $T$ and $\underline{q}_T\in\Xgrad{k}{T}$ is $\underline{q}_T=(q_T,(q_F)_{F\in\FT},(q_E)_{E\in\ET},(q_V)_{V\in\VT})$.

SDDR spaces differ from the DDR spaces introduced in \cite{Di-Pietro.Droniou:23} when $\ell_Y < k - 1$ for some $Y \in \Th \cup \Fh$.
The same degree of polynomial consistency is preserved through the following serendipity operators, that reconstruct, respectively, consistent gradient and vector potentials:
For all $Y\in\Th\cup\Fh$,
\[
\begin{gathered}
  \text{%
    $\SG{Y} : \Xgrad{k}{Y} \to \vPoly{k}(Y)$
    such that $\SG{Y}\Igrad{k}{Y}q = \GRAD_Y q$ for all $q \in \Poly{k+1}(Y)$,
  }
  \\
  \text{%
    $\SC{Y} : \Xcurl{k}{Y} \to \vPoly{k}(Y)$
    such that $\SC{Y}\Icurl{k}{Y}\bvec{v} = \bvec{v}$ for all $\bvec{v} \in \vPoly{k}(Y)$.
  }
\end{gathered}
\]

\subsubsection{Gradient space}

  Based on the vertex and edge components of a given $\underline{q}_h \in \Xgrad{k}{h}$, we can construct a continuous function $q_{\Eh}$ on the mesh edge skeleton whose restriction $q_{\Eh|E} \in \Poly{k+1}(E)$ to an edge $E \in \Eh$ is the unique polynomial that takes the value $q_V$ at each vertex $V \in \VE$ and satisfies $\lproj{k-1}{E} q_{\Eh|E} = q_E$.

For any mesh face $F \in \Fh$, we define the face gradient $\cGF{k} : \Xgrad{k}{F} \to \vPoly{k}(F)$ and the scalar trace $\trF : \Xgrad{k}{F} \to \Poly{k+1}(F)$ such that, for all $\underline{q}_F \in \Xgrad{k}{F}$,
\[
\begin{gathered}
  \int_F\cGF{k}\underline{q}_F\cdot(\bvec{w}+\bvec{\tau})
  = \sum_{E\in\EF}\omega_{FE}\int_E q_{\Eh} (\bvec{w}\cdot\normal_{FE})
  +\int_F \SG{F}\underline{q}_F\cdot\bvec{\tau}\quad\forall (\bvec{w},\bvec{\tau})\in\Roly{k}(F)\times\cRoly{k}(F),
  \\
  \int_F \trF\underline{q}_F\DIV_F \bvec{w}=-\int_F \cGF{k}\underline{q}_F\cdot\bvec{w}+\sum_{E\in\EF}\omega_{FE}\int_E \trE\underline{q}_E (\bvec{w}\cdot\normal_{FE})\quad\forall\bvec{w}\in\cRoly{k+2}(F).
\end{gathered}
\]
Similarly, for any mesh element $T \in \Th$, the element gradient $\cGT{k} : \Xgrad{k}{T} \to \vPoly{k}(T)$ and the corresponding potential $\Pgrad{k}{T} : \Xgrad{k}{T} \to \Poly{k+1}(T)$ satisfy, for all $\underline{q}_T \in \Xgrad{k}{T}$,
\[
\begin{gathered}
  \int_T\cGT{k}\underline{q}_T\cdot(\bvec{w}+\bvec{\tau})=\sum_{F\in\FT}\omega_{TF}\int_F \trF\underline{q}_F(\bvec{w}\cdot\normal_{TF})
  +\int_T\SG{T}\underline{q}_T\cdot\bvec{\tau}\quad\forall(\bvec{w},\bvec{\tau})\in\Roly{k}(T)\times\cRoly{k}(T),
  \\
  \int_T \Pgrad{k}{T}\underline{q}_T\DIV \bvec{w}=-\int_T \cGT{k}\underline{q}_T\cdot\bvec{w}+\sum_{F\in\FT}\omega_{TF}\int_F \trF\underline{q}_F (\bvec{w}\cdot\normal_{TF})\quad\forall\bvec{w}\in\cRoly{k+2}(T).
\end{gathered}
\]
  The global gradient $\uGh{k} : \Xgrad{k}{h} \to \Xcurl{k}{h}$ is such that, for all $\underline{q}_h \in \Xgrad{k}{h}$,
  \[
  \uGh{k}\underline{q}_h \coloneq \big(
  (\Rproj{k-1}{T}\cGT{k}\underline{q}_T, \Rcproj{\ell_T+1}{T}\cGT{k}\underline{q}_T)_{T \in \Th},
  (\Rproj{k-1}{F}\cGF{k}\underline{q}_F, \Rcproj{\ell_F+1}{F}\cGF{k}\underline{q}_F)_{F \in \Fh},
  (q_{\Eh|E}')_{E \in \Eh}
  \big).
  \]

\subsubsection{Curl space}

For any mesh face $F \in \Fh$, the face curl $\CF{k}:\Xcurl{k}{F}\to\Poly{k}(F)$ and the corresponding tangential trace $\trFt:\Xcurl{k}{F}\to\vPoly{k}(F)$ are such that, for all $\uvec{v}_F \in \Xcurl{k}{F}$,
\[
\begin{gathered}
  \int_F \CF{k}\uvec{v}_F r
  = \int_F \bvec{v}_{\cvec{R},F}\VROT_F r - \sum_{E\in\EF}\omega_{FE}\int_E v_E r\qquad\forall r\in\Poly{k}(F),
  \\
  \begin{multlined}
    \int_F \trFt\uvec{v}_F\cdot(\VROT_F r+\bvec{\tau})=\int_F \CF{k}\uvec{v}_F r + \sum_{E\in\EF}\omega_{FE}\int_E v_E r
    + \int_F \SC{F}\uvec{v}_F\cdot\bvec{\tau}
    \\
    \forall (r,\bvec{\tau})\in \Poly{0,k+1}(F)\times\cRoly{k}(F).
  \end{multlined}
\end{gathered}
\]
For all $T \in \Th$, the element curl $\cCT{k} : \Xcurl{k}{T} \to \vPoly{k}(T)$ and the corresponding vector potential $\Pcurl{k}{T} : \Xcurl{k}{T} \to \vPoly{k}(T)$ are such that, for all $\uvec{v}_T \in \Xcurl{k}{T}$,
\[
\begin{gathered}
  \int_T \cCT{k}\uvec{v}_T\cdot\bvec{w}=\int_T \bvec{v}_{\cvec{R},T}\cdot \CURL\bvec{w}+\sum_{F\in\FT}\omega_{TF}\int_F \trFt\uvec{v}_F\cdot(\bvec{w}\times\normal_F)\qquad\forall \bvec{w}\in\vPoly{k}(T)
  \\
  \begin{multlined}
    \int_T \Pcurl{k}{T}\uvec{v}_T\cdot(\CURL\bvec{w}+\bvec{\tau})
    = \int_T \cCT{k}\uvec{v}_T\cdot\bvec{w}
    - \sum_{F\in\FT}\omega_{TF}\int_F\trFt{}\uvec{v}_F\cdot(\bvec{w}\times\normal_F)
    + \int_T \SC{T}\uvec{v}_T\cdot\bvec{\tau}
    \\
    \forall (\bvec{w},\bvec{\tau})\in\cGoly{k+1}(T)\times\cRoly{k}(T).
  \end{multlined}
\end{gathered}
\]
The global curl $\uCh{k} : \Xcurl{k}{h} \to \Xdiv{k}{h}$ is such that, for all $\uvec{v}_h \in \Xcurl{k}{h}$,
\[
\uCh{k}\uvec{v}_h \coloneq \big(
(\Gproj{k-1}{T}\cCT{k}\uvec{v}_T, \Gcproj{k}{T}\cCT{k}\uvec{v}_T)_{T \in \Th},
(\CF{k}\uvec{v}_F)_{F \in \Fh}
\big).
\]

\subsubsection{Divergence space}

For all $T \in \Th$, the discrete divergence $\DT : \Xdiv{k}{T} \to \Poly{k}(T)$ and the corresponding vector potential $\Pdiv{k}{T} : \Xdiv{k}{T} \to \vPoly{k}(T)$ are such that, for all $\uvec{w}_T \in \Xdiv{k}{T}$,
\[
\begin{gathered}
  \int_T \DT\uvec{w}_T q
  = -\int_T\bvec{w}_{\cvec{G},T}\cdot\GRAD q
  + \sum_{F\in\FT}\omega_{TF}\int_F w_F q\quad\forall q\in\Poly{k}(T),
  \\
  \begin{multlined}
    \int_T \Pdiv{k}{T}\uvec{w}_T\cdot(\GRAD r+\bvec{\tau})
    = -\int_T \DT\uvec{w}_T r+\sum_{F\in\FT}\omega_{TF}\int_F w_F r
    + \int_T \bvec{w}_{\cvec{G},T}^\compl \cdot\bvec{\tau}
    \\
    \forall (r,\bvec{\tau})\in\Poly{0,k+1}(T)\times \cGoly{k}(T).
  \end{multlined}
\end{gathered}
\]
For future use, we notice that
\begin{equation}\label{eq:cCh}
  \cCh{k} = \Pdiv{k}{h} \circ \uCh{k}.
\end{equation}

\subsubsection{Discrete $L^2$-products}

For $\bullet \in \{\GRAD,\CURL,\DIV\}$, we define a discrete $L^2$-product on $\Xbullet{k}{h}$ setting, for all $\underline{x}_h, \underline{y}_h \in \Xbullet{k}{h}$,
\[
(\underline{x}_h, \underline{y}_h)_{\bullet,h}
\coloneq \sum_{T \in \Th} (\underline{x}_T, \underline{y}_T)_{\bullet,T}
\]
with, for all $T \in \Th$,
\[
(\underline{x}_T, \underline{y}_T)_{\bullet,T}
\coloneq \int_T \Pbullet{k}{T}\underline{x}_T \cdot \Pbullet{k}{T}\underline{y}_T
+ s_{\bullet,T}(\underline{x}_T - \Ibullet{k}{T}\Pbullet{k}{T}\underline{x}_T,
\underline{y}_T - \Ibullet{k}{T}\Pbullet{k}{T}\underline{y}_T),
\]
where the first term is responsible for consistency, while $s_{\bullet,T}$ is any positive semi-definite bilinear form that scales in $h$ like the consistency term and such that $(\cdot,\cdot)_{\bullet,T}$ is positive definite on $\Xbullet{k}{T}$.
Possible choices of stabilisation forms, used in the numerical simulations of Section \ref{sec:num.tests}, are the following:
\begin{align*}
  \mathrm{s}_{\GRAD,T}(\underline{r}_T,\underline{q}_T)
  \coloneq{}
  &\sum_{F\in\FT}h_F\int_F\big(\Pgrad{k+1}{T}\underline{r}_T-\trF\underline{r}_F\big) \big(\Pgrad{k+1}{T}\underline{q}_T-\trF\underline{q}_F\big)
  \nonumber\\
  & + \sum_{E\in\ET}h_E^2\int_E \big(\Pgrad{k+1}{T}\underline{r}_T-(r_{\Eh})_{|E}\big) \big(\Pgrad{k+1}{T}\underline{q}_T-(q_{\Eh})_{|E}\big),
\end{align*}
\begin{align*} 
  \mathrm{s}_{\CURL,T}(\uvec{w}_T,\uvec{v}_T)\coloneq{}
  &
  \sum_{F\in\FT}h_F\int_F \big( (\Pcurl{k}{T}\uvec{w}_T)_{{\rm t},F}-\trFt\uvec{w}_F\big)\cdot\big( (\Pcurl{k}{T}\uvec{v}_T)_{{\rm t},F}-\trFt\uvec{v}_F\big)
  \nonumber\\
  & + \sum_{E\in\ET}h_E^2\int_E\big(\Pcurl{k}{T}\uvec{w}_T\cdot\tangent_E-w_E\big)\big(\Pcurl{k}{T}\uvec{v}_T\cdot\tangent_E-v_E\big),
\end{align*}
(recall that ``${\rm t},F$'' denotes the tangential trace on $F$), and
\begin{equation*}
  \mathrm{s}_{\DIV,T}(\uvec{w}_T,\uvec{v}_T)
  \coloneq\sum_{F\in\FT}h_F\int_F\big(\Pdiv{k}{T}\uvec{w}_T\cdot\normal_F-w_F\big)\big(\Pdiv{k}{T}\uvec{v}_T\cdot\normal_F-v_F\big).
\end{equation*}

\subsection{SDDR scheme for the Navier--Stokes equations}

Denote by $\Xgrads{k}{h}$ the subspace of $\Xgrad{k}{h}$ embedding the zero-average condition, i.e.,
\[
\Xgrads{k}{h}
\coloneq\left\{
\underline{q}_h \in \Xgrad{k}{h} \st
(\underline{q}_h, \Igrad{k}{h}1)_{\GRAD,h} = 0
\right\}.
\]
We assume from here on that
\begin{equation}\label{eq:assum.f}
\bvec{f}\in \bvec{C}^0(\overline{\Omega}),
\end{equation}
so that $\Icurl{k}{h}\bvec{f}$ makes sense.
The SDDR scheme reads:
\begin{subequations}\label{eq:ddr}
  \begin{align}\nonumber
    &\text{%
      Find $(\uvec{u}_h,\underline{p}_h)\in \Xcurl{k}{h}\times \Xgrads{k}{h}$ such that, for all $(\uvec{v}_h,\underline{q}_h)\in \Xcurl{k}{h}\times \Xgrads{k}{h}$,
    }
    \\
    \label{eq:ddr.momentum}
    &\nu(\uCh{k}\uvec{u}_h,\uCh{k}\uvec{v}_h)_{\DIV,h}
    + \int_\Omega [\cCh{k} \uvec{u}_h\times \Pcurl{k}{h}\uvec{u}_h]\cdot\Pcurl{k}{h}\uvec{v}_h
    +(\uGh{k}\underline{p}_h,\uvec{v}_h)_{\CURL,h}
    = (\Icurl{k}{h}\bvec{f},\uvec{v}_h)_{\CURL,h},
    \\
    \label{eq:ddr.incompressibility}
    &-(\uvec{u}_h,\uGh{k}\underline{q}_h)_{\CURL,h} = 0.
  \end{align}
\end{subequations}

\subsubsection{Relevant constants}

To state the convergence theorem, we need to define several constants.

Since the topology of $\Omega$ is trivial, the SDDR complex is exact (see \cite[Theorem 2]{Di-Pietro.Droniou:23} and \cite[Section 6.6]{Di-Pietro.Droniou:23*1}), and thus $(\Image\uGh{k})^\perp=(\ker\uCh{k})^\perp$, where the orthogonal is considered for the inner product $(\cdot,\cdot)_{\CURL,h}$ on $\Xcurl{k}{h}$. Invoking the  Poincar\'e inequality \cite[Theorem 4]{Di-Pietro.Droniou:23} (together with \cite[Section 6.6]{Di-Pietro.Droniou:23*1}), we infer the existence of $\Cpoin$ not depending on $h$ such that
\begin{equation}\label{def:CP}
\norm{\CURL,h}{\uvec{v}_h}\le \Cpoin\norm{\DIV,h}{\uCh{k}\uvec{v}_h}\quad\forall \uvec{v}_h\in (\Image\uGh{k})^\perp.
\end{equation}

The inequalities \eqref{eq:bound.P} and \eqref{eq:equiv.norms} ensure the existence of the two continuity constants $\CPcurl$ and $\CPdiv$, independent of $h$, for the potential reconstructions:
\begin{alignat}{2}\label{def:CcPdiv}
  \norm{\bvec{L}^2(\Omega)}{\Pcurl{k}{h}\uvec{v}_h}
  &\le\CPcurl\norm{\CURL,h}{\uvec{v}_h}
  &\qquad& \forall \uvec{v}_h\in\Xcurl{k}{h},\\
  \label{def:CcPcurl}
  \norm{\bvec{L}^2(\Omega)}{\Pdiv{k}{h}\uvec{w}_h}
  &\le \CPdiv\norm{\DIV,h}{\uvec{w}_h}
  &\qquad& \forall \uvec{w}_h\in\Xdiv{k}{h}.
\end{alignat}
Since $\Cpoin$, $\CPcurl$, and $\CPdiv$ can be bounded from above by constants depending only on the mesh regularity parameter and $\Omega$, we will possibly include them in the hidden constant in the notation $\lesssim$.

We also need to define the following discrete Sobolev constant:
\begin{equation}\label{def:Csobh}
\begin{aligned}
  \Csobh \coloneq \max\Bigg\{\frac{\norm{\bvec{L}^4(\Omega)}{\Pcurl{k}{h}\uvec{v}_h}}{\norm{\DIV,h}{\uCh{k}\uvec{v}_h}}\,:\,{}&\uvec{v}_h\in(\Image\uGh{k})^\perp\backslash\{\bvec{0}\}\Bigg\}.
\end{aligned}
\end{equation}

\begin{remark}[Discrete Sobolev embedding]
  Establishing error estimates on schemes for the Navier--Stokes problem requires a Sobolev embedding from the velocity space into (at least) $\bvec{L}^4(\Omega)$. When considering a Laplacian-based formulation, the velocity space is contained in $\bvec{H}^1(\Omega)$ and Sobolev embeddings are available, including for discrete versions of this space; see, e.g. \cite[Appendix B]{Droniou.Eymard.ea:18} or \cite[Theorems 6.5 and 6.40]{Di-Pietro.Droniou:20}. For the formulation \eqref{eq:weak}, the velocity is only in $\Hcurl{\Omega}\cap\Hdiv{\Omega}$ (with a zero divergence and suitable boundary conditions); Sobolev embeddings are then much more challenging and, for polytopal domains, essentially require the convexity of the domain to ensure an $\bvec{H}^1$ regularity on the velocity \cite[Theorem 2.17]{Amrouche.Bernardi.ea:98}. Discrete Sobolev embeddings have been established for N\'ed\'elec finite elements \cite{Girault:90}, drawing strongly on the inclusion of the discrete space into $\Hcurl{\Omega}$, but to our knowledge still remain a largely open question for polytopal numerical methods.

  As a result of the exactness of the SDDR complex, the constant $\Csobh$ is finite for each $h$, but a proper Sobolev embedding for the SDDR method would require to prove that it remains bounded independently of $h$. Given the similarity (at least for $k=0$) on tetrahedral/hexahedral meshes between the SDDR complex and the Lagrange--N\'ed\'elec--Raviart--Thomas finite element complex, and accounting for the embedding proved for the latter in \cite{Girault:90}, we believe that $\Csobh$ is indeed bounded with respect to $h$, at least on quasi-uniform meshes.
  This conjecture seems to be confirmed by the orders of convergence observed in the numerical tests of Section \ref{sec:num.tests}.
  A rigorous proof of this bound is postponed to a future work, the main novelty of this paper being a new pressure-robust scheme for the Navier--Stokes equations supporting general polytopal meshes.
  In the rest of the paper, we always make the dependency on $\Csobh$ explicit (i.e., it is not hidden in the constants appearing in $\lesssim$).
\end{remark}

\subsubsection{Discrete norms}

  The error on the velocity will be measured in the following graph norm on $\Xcurl{k}{h}$:
  \begin{equation}\label{eq:norm.U.h}
    \norm{\bvec{U},h}{\uvec{v}_h}
    \coloneq \left(
    \norm{\CURL,h}{\uvec{v}_h}^2 + \norm{\DIV,h}{\uCh{k}\uvec{v}_h}^2
    \right)^{\frac12},
  \end{equation}
  and we denote the corresponding dual norm by $\norm{\bvec{U},h,*}{{\cdot}}$.

  For the error on the pressure, on the other hand, we use a $W^{1,\frac43}$-like norm.
As shown in \cite{Girault:90}, it is indeed expected that the pressure for the continuous model belongs to $W^{1,\frac43}(\Omega)$. The proof of this result isn't, however, based on direct estimates obtained through suitable test functions, and \cite{Girault:90} does not establish error estimates on the approximation of the pressure in this space, only in $L^2(\Omega)$ (through a duality argument that strongly relies on the conformity of the considered finite element method).
For all $s\in [1,\infty)$, we define the $L^s$-like norm of $\uvec{w}_h\in\Xcurl{k}{h}$ by setting
  \begin{equation}\label{eq:def.Ls.Xcurl}
    \begin{gathered}
      \norm{s,\CURL,h}{\uvec{w}_h} \coloneq \left(\sum_{T\in\Th}\norm{s,\CURL,T}{\uvec{w}_T}^s\right)^{\frac1s}\\
      \text{with }
      \begin{aligned}[t]
        \norm{s,\CURL,T}{\uvec{w}_T}^s
        &\coloneq
        \norm{\bvec{L}^s(T)}{\Pcurl{k}{T}\uvec{w}_T}^{s}+\sum_{F\in\FT}h_F\norm{\bvec{L}^s(F)}{(\Pcurl{k}{T}\uvec{w}_T)_{{\rm t},F}-\trFt\uvec{w}_F}^s\\
        &\quad
        +\sum_{E\in\ET}h_E^2\norm{L^s(E)}{\Pcurl{k}{T}\uvec{w}_T\cdot\tangent_E-w_E}^s,
      \end{aligned}
    \end{gathered}
  \end{equation}
  where $(\Pcurl{k}{T}\uvec{w}_T)_{{\rm t},F}=\normal_F\times((\Pcurl{k}{T}\uvec{w}_T)\times\normal_F)$ is the tangential projection on $F$ of $\Pcurl{k}{T}\uvec{w}_T$. We note that $\norm{s,\CURL,h}{{\cdot}}$ is the equivalent of \eqref{eq:def.pot.norm} (for $l=1$) in the exterior calculus version of the complex and that, for $s=2$, it is actually the norm associated with the $L^2$-product $(\cdot,\cdot)_{\CURL,h}$. The discrete $W^{1,\frac43}$-norm on the pressure space is then defined by:
For all $\underline{r}_h\in\Xgrads{k}{h}$,
\begin{equation}\label{eq:W14/3-norm}
  \norm{P,h}{\underline{r}_h}
  \coloneq
  \max\left\{(\uGh{k}\underline{r}_h,\uGh{k}\underline{q}_h)_{\CURL,h}\,:\,
  \underline{q}_h\in \Xgrads{k}{h}\,,\;\norm{4,\CURL,h}{\uGh{k}\underline{q}_h}\le 1\right\}.
\end{equation}
Taking $\underline{q}_h=c\underline{r}_h$ with $c\ge 0$ scaled to ensure that $\norm{4,\CURL,h}{\uGh{k}\underline{q}_h}= 1$ shows that $\uGh{k}\underline{r}_h=0$ whenever $\norm{P,h}{\underline{r}_h}=0$ and thus, by the Poincar\'e inequality for $\Xgrads{k}{h}$ \cite[Theorem 3]{Di-Pietro.Droniou:23}, that $\norm{P,h}{{\cdot}}$ is indeed a norm on this space.

\subsubsection{Error estimates}

  Let the $\bvec{H}^{(k+1,2)}(\Th)$-seminorm be defined by: For all $\bvec{w}\in \bvec{H}^{\max(k+1,2)}(\Th)$,
  \[
  \begin{aligned}
    \seminorm{\bvec{H}^{(k+1,2)}(\Th)}{\bvec{w}}\coloneq{}&\left(\sum_{T\in\Th}\seminorm{\bvec{H}^{(k+1,2)}(T)}{\bvec{w}}^2\right)^{\frac12},\\
    \mbox{ with }
    \seminorm{\bvec{H}^{(k+1,2)}(T)}{\bvec{w}}\coloneq{}&\left\{
    \begin{array}{ll}
      \seminorm{\bvec{H}^1(T)}{\bvec{w}}+h_T\seminorm{\bvec{H}^2(T)}{\bvec{w}}&\mbox{ if $k=0$},\\
      \seminorm{\bvec{H}^{k+1}(T)}{\bvec{w}}& \mbox{ if $k\ge 1$.}
    \end{array}\right.
  \end{aligned}
  \]

The convergence theorem, proved in Section \ref{sec:analysis}, is the following.

\begin{theorem}[Error estimate for the SDDR scheme]\label{thm:convergence}
  Assume that \eqref{eq:weak} has a solution $(\bvec{u},p)\in\Hcurl{\Omega}\times H^1(\Omega)$ such that
  \begin{equation}\label{eq:regularity.u}
    \begin{aligned}
      \bvec{u}\in{}& \bvec{W}^{1,4}(\Omega)\cap \bvec{W}^{k+1,4}(\Th)\cap\bvec{H}^{\max(k+1,2)}(\Th),\\
      \CURL \bvec{u}\in{}& \bvec{C}^0(\overline{\Omega})\cap \bvec{H}(\CURL;\Omega)\cap \bvec{H}^{k+2}(\Th)\cap \bvec{W}^{k+1,4}(\Th),\\
      \CURL\bvec{u}\times\bvec{u}\in{}& \bvec{H}^{\max(k+1,2)}(\Th),\quad\CURL\CURL\bvec{u}\in \bvec{H}^{\max(k+1,2)}(\Th).
    \end{aligned}
  \end{equation}
  We denote by 
  \begin{equation}\label{eq:def.Ru}
    \bvec{R}_{\bvec{u}} \coloneq \nu\CURL\CURL\bvec{u}+\CURL\bvec{u}\times\bvec{u}
  \end{equation}
  the part of $\bvec{f}$ depending only on the velocity. 
  Under the mesh assumption of Section \ref{sec:mesh}, suppose further that
  \begin{equation}\label{eq:def.chi}
    \chi\coloneq \nu-\CPdiv\Csobh^2\Cpoin\nu^{-1}\norm{\CURL,h}{\Icurl{k}{h}\bvec{R}_{\bvec{u}}}>0.
  \end{equation}
  Then, if $(\uvec{u}_h,\underline{p}_h)$ is the solution to the SDDR scheme \eqref{eq:ddr}, the following
  error estimates hold:
  \begin{align}\label{eq:error.estimate.u}
    \norm{\bvec{U},h}{\uvec{u}_h-\Icurl{k}{h}\bvec{u}}
    &\lesssim \mathcal K_1(\bvec{u})\chi^{-1}h^{k+1},\\
    \label{eq:error.estimate.p}
    \norm{P,h}{\underline{p}_h-\Igrad{k}{h}p}
    &\lesssim \mathcal K_2(\bvec{u})\chi^{-1}h^{k+1},
  \end{align}
  where
  \[
  \begin{aligned}
    \mathcal K_1(\bvec{u})={}&\big\lbrack\nu(|\CURL\CURL\bvec{u}|_{\bvec{H}^{(k+1,2)}(\Th)}+|\CURL\bvec{u}|_{\bvec{H}^{k+1}(\Th)}+|\CURL\bvec{u}|_{\bvec{H}^{k+2}(\Th)})
    \\ \nonumber
    &\hphantom{\big\lbrack}
    + |\CURL \bvec{u}\times\bvec{u}|_{\bvec{H}^{(k+1,2)}(\Th)}
    +\seminorm{\bvec{W}^{k+1,4}(\Th)}{\CURL\bvec{u}}\norm{\bvec{L}^4(\Omega)}{\bvec{u}}
    +\seminorm{\bvec{W}^{1,4}(\Omega)}{\bvec{u}}\seminorm{\bvec{W}^{k+1,4}(\Th)}{\bvec{u}}\big\rbrack\\
    &+\Csobh\seminorm{\bvec{W}^{1,4}(\Omega)}{\bvec{u}}\seminorm{\bvec{H}^{(k+1,2)}(\Th)}{\bvec{u}}+\seminorm{\bvec{H}^{(k+1,2)}(\Th)}{\bvec{u}},\\
    \mathcal K_2(\bvec{u})={}&\left(1+\Csobh\nu^{-1}\norm{\CURL,h}{\Icurl{k}{h}\bvec{R}_{\bvec{u}}}+\seminorm{\bvec{W}^{1,4}(\Omega)}{\bvec{u}}\right)\mathcal K_1(\bvec{u}).
  \end{aligned}
  \]
  
\end{theorem}

  \begin{remark}[Data smallness assumption]
    Notice that, similarly to \cite[Eq.~(3.8)]{Linke.Merdon:16} or \cite[Eq.~(68)]{Castanon-Quiroz.Di-Pietro:20}, the data smallness assumption \eqref{eq:def.chi} only involves the solenoidal part of the forcing term, consistently with the fact that we aim at error estimates that are robust for large irrotational body forces.
  \end{remark}


\section{Analysis}\label{sec:analysis}

\subsection{A priori estimates and existence of a solution to the scheme}

A priori estimates on the velocity are straightforward, and do not depend on $\Csobh$.

\begin{lemma}[A priori estimate on the velocity]\label{lem:apriori.u}
  If $(\uvec{u}_h,\underline{p}_h)$ solves \eqref{eq:ddr}, then
  \begin{equation}\label{eq:est.u}
    \norm{\DIV,h}{\uCh{k}\uvec{u}_h}\le \Cpoin\nu^{-1}\norm{\CURL,h}{\Icurl{k}{h}\bvec{f}}\quad\mbox{ and }\quad
    \norm{\CURL,h}{\uvec{u}_h}\le \Cpoin^2\nu^{-1}\norm{\CURL,h}{\Icurl{k}{h}\bvec{f}}.
  \end{equation}
  If, moreover, the pressure in \eqref{eq:weak} satisfies $p_{|T}\in C^1(\overline{T})$ for all $T\in\Th$, then the terms $\bvec{f}$ above can be replaced with $\bvec{R}_{\bvec{u}}$ defined by \eqref{eq:def.Ru}, that is:
    \begin{equation}\label{eq:est.u.robust}
      \norm{\DIV,h}{\uCh{k}\uvec{u}_h}\le \Cpoin\nu^{-1}\norm{\CURL,h}{\Icurl{k}{h}\bvec{R}_{\bvec{u}}}\quad\mbox{ and }
      \quad\norm{\CURL,h}{\uvec{u}_h}\le \Cpoin^2\nu^{-1}\norm{\CURL,h}{\Icurl{k}{h}\bvec{R}_{\bvec{u}}}.
    \end{equation}
\end{lemma}

\begin{proof}
  Take $(\uvec{v}_h,\underline{q}_h)=(\uvec{u}_h,\underline{p}_h)$ as a test function in \eqref{eq:ddr}, use the fact that $\cCh{k}\uvec{u}_h\times\Pcurl{k}{h}\uvec{u}_h$ is orthogonal to $\Pcurl{k}{h}\uvec{u}_h$ and add together \eqref{eq:ddr.momentum} and \eqref{eq:ddr.incompressibility} to get
  \begin{equation}\label{eq:apriori.u}
    \nu\norm{\DIV,h}{\uCh{k}\uvec{u}_h}^2=(\Icurl{k}{h}\bvec{f},\uvec{u}_h)_{\CURL,h}\le \norm{\CURL,h}{\Icurl{k}{h}\bvec{f}}\norm{\CURL,h}{\uvec{u}_h}.
  \end{equation}
  The proof of \eqref{eq:est.u} is completed using \eqref{eq:ddr.incompressibility} to see that $\uvec{u}_h\in(\Image\uGh{k})^\perp$ and by invoking the discrete Poincar\'e inequality for the curl \eqref{def:CP}.

  Assume now that the pressure satisfies $p_{|T}\in C^1(\overline{T})$ for all $T\in\Th$.
  Then, the local commutation property \cite[Eq.~(3.38)]{Di-Pietro.Droniou:23} gives $\Icurl{k}{T}\GRAD p=\uGT{k}\Igrad{k}{T}p$.
  Moreover, since $p\in H^1(\Omega)$, $p$ is actually continuous on $\overline{\Omega}$, and the tangential components of its gradients on the edges and faces are also continuous; hence, these local commutation properties can be patched together and yield $\Icurl{k}{h}\GRAD p =\uGh{k}\Igrad{k}{h}p$. As a consequence, $\bvec{R}_{\bvec{u}}=\bvec{f}-\GRAD p$ also has continuous tangential components on the edges and faces, and $\Icurl{k}{h}\bvec{f}=\Icurl{k}{h}\bvec{R}_{\bvec{u}}+\uGh{k}\Igrad{k}{h}p$. In \eqref{eq:apriori.u}, we can therefore write
  \[
  (\Icurl{k}{h}\bvec{f},\uvec{u}_h)_{\CURL,h}=(\Icurl{k}{h}\bvec{R}_{\bvec{u}},\uvec{u}_h)_{\CURL,h}
  +\cancel{(\uGh{k}\Igrad{k}{h}p,\uvec{u}_h)_{\CURL,h}},
  \]
  the cancellation being ensured by \eqref{eq:ddr.incompressibility}. The estimate \eqref{eq:est.u.robust} can therefore be written with $\bvec{R}_{\bvec{u}}$ instead of $\bvec{f}$, which concludes the proof of \eqref{eq:est.u.robust}.
\end{proof}

The a priori estimates on the pressure are naturally done in the $W^{1,\frac43}$-like norm \eqref{eq:W14/3-norm}.

\begin{lemma}[A priori estimates on the pressure]
  If $(\uvec{u}_h,\underline{p}_h)$ solves \eqref{eq:ddr}, then 
  \begin{equation}\label{eq:est.gradp.dual}
    \norm{P,h}{\underline{p}_h}
    \lesssim
    \norm{\CURL,h}{\Icurl{k}{h}\bvec{f}}+\Csobh \norm{\DIV,h}{\uCh{k}\uvec{u}_h}^2.
  \end{equation}
\end{lemma}

\begin{proof}
  Take $\underline{q}_h\in \Xgrads{k}{h}$ and plug $\uvec{v}_h=\uGh{k}\underline{q}_h$ into \eqref{eq:ddr.momentum}. Since $\uCh{k}\uGh{k}=0$ by the complex property, this gives
  \[
  (\uGh{k}\underline{p}_h,\uGh{k}\underline{q}_h)_{\CURL,h} =(\Icurl{k}{h}\bvec{f},\uGh{k}\underline{q}_h)_{\CURL,h}
  - \int_\Omega [\Pdiv{k}{h}\uCh{k} \uvec{u}_h\times \Pcurl{k}{h}\uvec{u}_h]\cdot\Pcurl{k}{h}\uGh{k}\underline{q}_h,
  \]
  where we have additionally used the characterisation \eqref{eq:cCh} of $\cCh{k}$.
  We then apply a Cauchy--Schwarz inequality and a generalised H\"older inequality with exponents $(2,4,4)$ respectively to the first and second term in the right-hand side, to get
  \begin{align}
    (\uGh{k}\underline{p}_h,\uGh{k}\underline{q}_h)_{\CURL,h} \le{}& \norm{\CURL,h}{\Icurl{k}{h}\bvec{f}}\norm{\CURL,h}{\uGh{k}\underline{q}_h}\nonumber\\
    &+\norm{\bvec{L}^2(\Omega)}{\Pdiv{k}{h}\uCh{k} \uvec{u}_h}\norm{\bvec{L}^4(\Omega)}{\Pcurl{k}{h}\uvec{u}_h}\norm{\bvec{L}^4(\Omega)}{\Pcurl{k}{h}\uGh{k}\underline{q}_h}\nonumber\\
    \lesssim{}&
    \norm{\CURL,h}{\Icurl{k}{h}\bvec{f}}\norm{\CURL,h}{\uGh{k}\underline{q}_h}+\Csobh\norm{\DIV,h}{\uCh{k}\uvec{u}_h}^2\norm{4,\CURL,h}{\uGh{k}\underline{q}_h},
    \label{eq:est.pressure.1}
  \end{align}
  where the second line follows using \eqref{eq:bound.P} below with $(s,l,f)=(2,2,T)$ and $(s,l,f)=(4,1,T)$, the norm equivalence \eqref{eq:equiv.norms} to write $\norm{\bvec{L}^2(\Omega)}{\Pdiv{k}{h}\uCh{k} \uvec{u}_h}\lesssim \norm{\DIV,h}{\uCh{k}\uvec{u}_h}$ and $\norm{\bvec{L}^4(\Omega)}{\Pcurl{k}{h}\uGh{k}\underline{q}_h}\lesssim \norm{4,\CURL,h}{\uGh{k}\underline{q}_h}$, and the definition of $\Csobh$ together with \eqref{eq:ddr.incompressibility}
  to write $\norm{\bvec{L}^4(\Omega)}{\Pcurl{k}{h}\uvec{u}_h}\le\Csobh\norm{\DIV,h}{\uCh{k}\uvec{u}_h}$.

  The discrete Lebesgue inequality \eqref{eq:discrete.lebesgue} with $(s,t)=(2,4)$ and $l=1$ gives, for all $T\in\Th$,
  \begin{equation}\label{eq:est.G.4.2}
    \norm{\CURL,T}{\uGT{k}\underline{q}_T}\lesssim h_T^{\frac34}\norm{4,\CURL,T}{\uGT{k}\underline{q}_T}.
  \end{equation}
  Squaring, summing over $T\in\Th$, using the Cauchy--Schwarz inequality, and invoking the mesh regularity property to write $h_T^3\lesssim |T|$, we infer
  \[
  \norm{\CURL,h}{\uGh{k}\underline{q}_h}^2\lesssim \left(\sum_{T\in\Th}h_T^3\right)^{\frac12}\norm{4,\CURL,h}{\uGh{k}\underline{q}_h}^2
  \lesssim |\Omega|^{\frac12}\norm{4,\CURL,h}{\uGh{k}\underline{q}_h}^2.
  \]
  Take the square root and plug the resulting estimate into \eqref{eq:est.pressure.1} to obtain
  \[
  (\uGh{k}\underline{p}_h,\uGh{k}\underline{q}_h)_{\CURL,h}\lesssim \left(\norm{\CURL,h}{\Icurl{k}{h}\bvec{f}}+
  \Csobh\norm{\DIV,h}{\uCh{k}\uvec{u}_h}^2\right)\norm{4,\CURL,h}{\uGh{k}\underline{q}_h}.
  \]
  Taking the maximum over $\underline{q}_h\in\Xgrads{k}{h}$ such that $\norm{4,\CURL,h}{\uGh{k}\underline{q}_h}\le 1$ concludes the proof.
\end{proof}

\subsection{Bound on the consistency error}

We define the consistency errors $\mathcal E_{h,1}(\bvec{u},p;\cdot):\Xcurl{k}{h}\to\Real$ and $\mathcal E_{h,2}(\bvec{u},p;\cdot):\Xgrad{k}{h}\to\Real$ by setting, for all $(\uvec{v}_h, \underline{q}_h) \in \Xcurl{k}{h} \times \Xgrad{k}{h}$, 
\begin{align}
  \label{eq:def.Eh.u}
  \mathcal E_{h,1}(\bvec{u},p;\uvec{v}_h)
  &\coloneq
  \begin{aligned}[t]
    &(\Icurl{k}{h}\bvec{f},\uvec{v}_h)_{\CURL,h}-\nu(\uCh{k}\Icurl{k}{h}\bvec{u},\uCh{k}\uvec{v}_h)_{\DIV,h}
    \\
    &\quad
    -\int_\Omega [\cCh{k}\Icurl{k}{h}\bvec{u}\times\Pcurl{k}{h}\Icurl{k}{h}\bvec{u}]\cdot\Pcurl{k}{h}\uvec{v}_h
    -(\uGh{k}\Igrad{k}{h}p,\uvec{v}_h)_{\CURL,h},
  \end{aligned}
  \\
  \label{eq:def.Eh.p}
  \mathcal E_{h,2}(\bvec{u},p;\underline{q}_h)
  &\coloneq
  (\Icurl{k}{h}\bvec{u},\uGh{k}\underline{q}_h)_{\CURL,h}.
\end{align}

We endow $\Xcurl{k}{h}$ with the norm $\norm{\bvec{U},h}{{\cdot}}$ defined by \eqref{eq:norm.U.h}. The space $\Xgrads{k}{h}$ is endowed with the norm $\norm{\bvec{G},h}{{\cdot}} \coloneq \norm{\CURL,h}{\uGh{k}{\cdot}}$ and the corresponding dual norm is denoted by $\norm{\bvec{G},h,*}{{\cdot}}$.

\begin{lemma}[Consistency bounds]\label{lem:consistency.bounds}
  We have the following bounds on the consistency errors:
  For all $\bvec{u}$ satisfying \eqref{eq:regularity.u} together with the boundary conditions $\CURL\bvec{u}\times\normal=\bvec{0}$ and $\bvec{u}\cdot\normal=0$ on $\partial\Omega$, it holds
  \begin{alignat}{2}\nonumber
    \norm{\bvec{U},h,*}{\mathcal E_{h,1}(\bvec{u},p;\cdot)}
    &\lesssim h^{k+1}\big\lbrack\nu(|\CURL\CURL\bvec{u}|_{\bvec{H}^{(k+1,2)}(\Th)}+|\CURL\bvec{u}|_{\bvec{H}^{k+1}(\Th)}+|\CURL\bvec{u}|_{\bvec{H}^{k+2}(\Th)})
    \\ \nonumber
    &\hphantom{\lesssim h^{k+1}\big\lbrack}
    + |\CURL \bvec{u}\times\bvec{u}|_{\bvec{H}^{(k+1,2)}(\Th)}
    +\seminorm{\bvec{W}^{k+1,4}(\Th)}{\CURL\bvec{u}}\norm{\bvec{L}^4(\Omega)}{\bvec{u}}
    \\
    &\hphantom{\lesssim h^{k+1}\big\lbrack}
    +\seminorm{\bvec{W}^{1,4}(\Omega)}{\bvec{u}}\seminorm{\bvec{W}^{k+1,4}(\Th)}{\bvec{u}}\big\rbrack,
    \label{eq:est.Eh.u}\\
    \label{eq:est.Eh.p}
    \norm{\bvec{G},h,*}{\mathcal E_{h,2}(\bvec{u},p;\cdot)}
    &\lesssim h^{k+1}|\bvec{u}|_{\bvec{H}^{(k+1,2)}(\Th)}.
  \end{alignat}
\end{lemma}

\begin{proof}
  We first prove \eqref{eq:est.Eh.u}.
  The assumption \eqref{eq:regularity.u} gives $\CURL\bvec{u}\times\bvec{u}\in \bvec{C}^0(\overline{\Omega})$ and $(\nu\CURL\CURL\bvec{u})_{|T}\in \bvec{H}^2(T)\subset \bvec{C}^0(\overline{T})$ for all $T\in\Th$. Combined with \eqref{eq:assum.f}, this shows that $(\GRAD p)_{|T}\in \bvec{C}^0(\overline{T})$ for all $T\in\Th$. Reasoning as in the proof of Lemma \ref{lem:apriori.u}, we can therefore split $\Icurl{k}{h}\bvec{f}$ and use the commutation properties to write 
  \[
  \Icurl{k}{h}\bvec{f}
  = \nu\Icurl{k}{h}(\CURL\CURL\bvec{u})+\Icurl{k}{h}(\CURL\bvec{u}\times\bvec{u})+\uGh{k}(\Igrad{k}{h}p).
  \]
  We therefore obtain, for all $\uvec{v}_h\in\Xcurl{k}{h}$,
  \begin{align*}
    \mathcal E_{h,1}(\bvec{u},p;\uvec{v}_h)
    &=\nu(\Icurl{k}{h}(\CURL\CURL\bvec{u}),\uvec{v}_h)_{\CURL,h}-\nu(\uCh{k}\Icurl{k}{h}\bvec{u},\uCh{k}\uvec{v}_h)_{\DIV,h}\\
    &\quad
    +(\Icurl{k}{h}(\CURL \bvec{u} \times \bvec{u}),\uvec{v}_h)_{\CURL,h}-\int_\Omega [\cCh{k}\Icurl{k}{h}\bvec{u}\times\Pcurl{k}{h}\Icurl{k}{h}\bvec{u}]\cdot\Pcurl{k}{h}\uvec{v}_h\\
    &=
    \nu\left[(\Icurl{k}{h}(\CURL\CURL\bvec{u}),\uvec{v}_h)_{\CURL,h}-(\Idiv{k}{h}\CURL\bvec{u},\uCh{k}\uvec{v}_h)_{\DIV,h}\right]\\
    &\quad
    +(\Icurl{k}{h}(\CURL \bvec{u} \times \bvec{u}),\uvec{v}_h)_{\CURL,h}-\int_\Omega [\cCh{k}\Icurl{k}{h}\bvec{u}\times\Pcurl{k}{h}\Icurl{k}{h}\bvec{u}]\cdot\Pcurl{k}{h}\uvec{v}_h
  \end{align*}
  where, for the second equality, we have used the commutation property \cite[Eq.~(3.39)]{Di-Pietro.Droniou:23} (which extends to the SDDR sequence, see \cite[Proposition 8]{Di-Pietro.Droniou:23*1}) to write $\uCh{k}\Icurl{k}{h}\bvec{u}=\Idiv{k}{h}\CURL\bvec{u}$ in the second term of the first line.
  
  We find a bound for the linear terms first. Introducing the integral $\int_\Omega\CURL\CURL\bvec{u}\cdot\Pcurl{k}{h}\uvec{v}_h$ and splitting into two differences, then using the primal and adjoint consistency \cite[Eq.~(6.12) and Theorem 10]{Di-Pietro.Droniou:23} to estimate each term, and recalling the definition \eqref{eq:norm.U.h} of $\norm{\bvec{U},h}{{\cdot}}$, we get
  \[
  \begin{aligned}
    &\left|(\Icurl{k}{h}(\CURL\CURL\bvec{u}),\uvec{v}_h)_{\CURL,h}
    - (\Idiv{k}{h}\CURL\bvec{u},\uCh{k}\uvec{v}_h)_{\DIV,h}\right|
    \\
    &\qquad
    \leq
    \left|(\Icurl{k}{h}(\CURL\CURL\bvec{u}),\uvec{v}_h)_{\CURL,h}-\int_\Omega\CURL\CURL\bvec{u}\cdot\Pcurl{k}{h}\uvec{v}_h\right|
    \\
    &\qquad\quad
    +\left|\int_\Omega\CURL\CURL\bvec{u}\cdot\Pcurl{k}{h}\uvec{v}_h-(\Idiv{k}{h}\CURL\bvec{u},\uCh{k}\uvec{v}_h)_{\DIV,h}\right|
    \\
    &\qquad
    \lesssim
    h^{k+1}|\CURL\CURL\bvec{u}|_{\bvec{H}^{(k+1,2)}(\Th)}\norm{\CURL,h}{\uvec{v}_h}
    \\
    &\qquad\quad
    + h^{k+1}\left(
    |\CURL\bvec{u}|_{\bvec{H}^{k+1}(\Th)}
    + |\CURL\bvec{u}|_{\bvec{H}^{k+2}(\Th)}
    \right)\norm{\bvec{U},h}{\uvec{v}_h}.
  \end{aligned}
  \]
  
  Next we deal with the nonlinear terms in the same manner, by adding and subtracting successively $\int_\Omega (\CURL \bvec{u}\times \bvec{u})\cdot\Pcurl{k}{h}\uvec{v}_h$ and $\int_\Omega (\cCh{k}\Icurl{k}{h}\bvec{u}\times \bvec{u})\cdot\Pcurl{k}{h}\uvec{v}_h$, and using again \cite[Eq.~(6.12)]{Di-Pietro.Droniou:23} in the first term and generalised H\"older inequalities in the remaining terms to obtain
  \begin{align}
    &\bigg|(\Icurl{k}{h}(\CURL \bvec{u} \times \bvec{u}),\uvec{v}_h)_{\CURL,h}
    - \int_\Omega (\cCh{k}\Icurl{k}{h}\bvec{u}\times\Pcurl{k}{h}\Icurl{k}{h}\bvec{u})\cdot\Pcurl{k}{h}\uvec{v}_h\bigg|
    \nonumber\\
    &\quad
    \leq\left|(\Icurl{k}{h}(\CURL \bvec{u} \times \bvec{u}),\uvec{v}_h)_{\CURL,h}
    - \int_\Omega (\CURL \bvec{u}\times \bvec{u})\cdot\Pcurl{k}{h}\uvec{v}_h\right|
    \nonumber\\
    &\qquad
    + \left|\int_\Omega \left[(\CURL \bvec{u} - \cCh{k}\Icurl{k}{h}\bvec{u})\times \bvec{u}\right]\cdot\Pcurl{k}{h}\uvec{v}_h
    \right|
    \nonumber\\
    &\qquad
    + \left|\int_\Omega \left[
      \cCh{k}\Icurl{k}{h}\bvec{u}\times (\bvec{u} - \Pcurl{k}{h}\Icurl{k}{h}\bvec{u})
      \right]
      \cdot\Pcurl{k}{h}\uvec{v}_h
    \right|
    \nonumber\\
    &\quad
    \lesssim
    h^{k+1}|\CURL \bvec{u}\times \bvec{u}|_{\bvec{H}^{(k+1,2)}(\Th)}\norm{\CURL,h}{\uvec{v}_h}
    \nonumber\\
    &\qquad
    + \norm{\bvec{L}^4(\Omega)}{\CURL \bvec{u}-\cCh{k}\Icurl{k}{h}\bvec{u}}\norm{\bvec{L}^4(\Omega)}{\bvec{u}}\norm{\bvec{L}^2(\Omega)}{\Pcurl{k}{h}\uvec{v}_h}
    \nonumber\\
    &\qquad
    +\norm{\bvec{L}^4(\Omega)}{\cCh{k}\Icurl{k}{h}\bvec{u}}\norm{\bvec{L}^4(\Omega)}{\bvec{u}-\Pcurl{k}{h}\Icurl{k}{h}\bvec{u}}\norm{\bvec{L}^2(\Omega)}{\Pcurl{k}{h}\uvec{v}_h}.  
  \label{est:nonlinear.term}
  \end{align}
  We first notice that
  \begin{equation*}
    \cCT{k}\Icurl{k}{T}\bvec{w}
    \overset{\eqref{eq:cCh}}=\Pdiv{k}{T}\uCT{k}\Icurl{k}{T}\bvec{w}
    = \Pdiv{k}{T}\Idiv{k}{T}\CURL\bvec{w},
  \end{equation*}
  where the second equality is a consequence of \cite[Proposition 8 and Section 6.6]{Di-Pietro.Droniou:23*1} and \cite[Eq.~(3.39)]{Di-Pietro.Droniou:23} (for $\bvec{w}\in \bvec{H}^2(T)$, the adaptation to $\bvec{w}\in \bvec{W}^{1,4}(T)$ such that $\CURL\bvec{w}\in \bvec{W}^{1,4}(T)$ being straightforward).
  Then, $\norm{\bvec{L}^4(\Omega)}{\CURL \bvec{u}-\cCh{k}\Icurl{k}{h}\bvec{u}}=\norm{\bvec{L}^4(\Omega)}{\CURL \bvec{u}-\Pdiv{k}{h}\Idiv{k}{h}\CURL\bvec{u}}$.
  We note that, for each $T\in\Th$, the local operators $\Pcurl{k}{T}\circ\Icurl{k}{T}$ and $\Pdiv{k}{T}\circ\Idiv{k}{T}$ are projections by \cite[Proposition 1.35]{Di-Pietro.Droniou:20} since they reproduce exactly polynomials up to degree $k$, cf. \cite[Eqs.~(4.7) and (4.12)]{Di-Pietro.Droniou:23}. The estimate \eqref{eq:bound.PI} below (with $m=0$, $l=1,2$ and $s=4$) give a bound on these projections, which enables us to get the local approximation properties stated in \cite[Lemma 1.43]{Di-Pietro.Droniou:20} applied to $(p,l,s,q,m)=(4,k,k+1,1,0)$.
  These approximation properties yield
  \begin{align*}
    \norm{\bvec{L}^4(\Omega)}{\CURL \bvec{u}-\Pdiv{k}{h}\Idiv{k}{h}\CURL\bvec{u}}{}&\lesssim h^{k+1}|\CURL\bvec{u}|_{\bvec{W}^{k+1,4}(\Th)},\\
    \norm{\bvec{L}^4(\Omega)}{\bvec{u}-\Pcurl{k}{h}\Icurl{k}{h}\bvec{u}}{}&\lesssim h^{k+1}|\bvec{u}|_{\bvec{W}^{k+1,4}(\Th)}.
  \end{align*}
  Thus, using \eqref{eq:bound.P}, \eqref{eq:equiv.norms} and \eqref{eq:bound.dI} (with $m=0$) below to bound $\norm{\bvec{L}^2(\Omega)}{\Pcurl{k}{h}\uvec{v}_h}$ and $\norm{\bvec{L}^4(\Omega)}{\cCh{k}\Icurl{k}{h}\bvec{u}}$, and plugging the estimates above into \eqref{est:nonlinear.term}, we obtain
  \begin{multline*}
\bigg|(\Icurl{k}{h}(\CURL \bvec{u} \times \bvec{u}),\uvec{v}_h)_{\CURL,h}
    - \int_\Omega (\cCh{k}\Icurl{k}{h}\bvec{u}\times\Pcurl{k}{h}\Icurl{k}{h}\bvec{u})\cdot\Pcurl{k}{h}\uvec{v}_h\bigg|\\
    \lesssim h^{k+1}\Big(|\CURL \bvec{u}\times\bvec{u}|_{\bvec{H}^{(k+1,2)}(\Th)}+ |\CURL\bvec{u}|_{\bvec{W}^{k+1,4}(\Th)}\norm{\bvec{L}^4(\Omega)}{\bvec{u}}+\seminorm{\bvec{W}^{1,4}(\Th)}{\bvec{u}}|\bvec{u}|_{\bvec{W}^{k+1,4}(\Th)}\Big)\norm{\CURL,h}{\uvec{v}_h}.
  \end{multline*}
  Gathering all the estimates we infer \eqref{eq:est.Eh.u}.

  The bound \eqref{eq:est.Eh.p} is a straightforward consequence of the adjoint consistency for the gradient \cite[Theorem 9]{Di-Pietro.Droniou:23}, where $\DIV\bvec{u}=0$.
\end{proof}

\subsection{Proof of the convergence theorem}

We start with the error estimate on the velocity. Let $\uvec{e}_h \coloneq \uvec{u}_h-\Icurl{k}{h}\bvec{u}\in\Xcurl{k}{h}$ and $\underline{\epsilon}_h \coloneq \underline{p}_h-\Igrad{k}{h}p\in\Xgrad{k}{h}$.
By definition \eqref{eq:def.Eh.u} of the first consistency error, we have, for all $\uvec{v}_h\in\Xcurl{k}{h}$,
\begin{align}
  \mathcal E_{h,1}{}&(\bvec{u},p;\uvec{v}_h)
  =
  \nu(\uCh{k}\uvec{e}_h,\uCh{k}\uvec{v}_h)_{\DIV,h}\nonumber\\
  &\quad
  +\int_\Omega \left(
  \cCh{k}\uvec{u}_h\times\Pcurl{k}{h}\uvec{u}_h - \cCh{k}\Icurl{k}{h}\bvec{u}\times\Pcurl{k}{h}\Icurl{k}{h}\bvec{u}
  \right)\cdot\Pcurl{k}{h}\uvec{v}_h+(\uGh{k}\underline{\epsilon}_h,\uvec{v}_h)_{\CURL,h}\nonumber\\
  &=\nu(\uCh{k}\uvec{e}_h,\uCh{k}\uvec{v}_h)_{\DIV,h}
  \nonumber\\
  &\quad
  + \int_\Omega \left(
  \cCh{k}\uvec{e}_h\times\Pcurl{k}{h}\uvec{u}_h + \cCh{k}\Icurl{k}{h}\bvec{u}\times\Pcurl{k}{h}\uvec{e}_h
  \right)\cdot\Pcurl{k}{h}\uvec{v}_h+(\uGh{k}\underline{\epsilon}_h,\uvec{v}_h)_{\CURL,h},
  \label{eq:error.bound.E1}
\end{align}
  where we have added and subtracted $\cCh{k}\Icurl{k}{h}\bvec{u} \times \Pcurl{k}{h}\uvec{u}_h$ inside parentheses to pass to the second line.
  We then write $\Icurl{k}{h}\bvec{u}=\uvec{w}_h^\sharp+\uvec{w}_h^\bot\in \Image\uGh{k}\oplus(\Image\uGh{k})^\bot$, where the orthogonal is taken for the inner product $(\cdot,\cdot)_{\CURL,h}$.
  Since $\uCh{k}\circ\uGh{k}=\uvec{0}$, we have
  \begin{equation}\label{eq:uCh.I.u=uCh.w}
    \uCh{k}\Icurl{k}{h}\bvec{u}=\uCh{k}\uvec{w}_h^\bot.
  \end{equation}
  Selecting $\uvec{v}_h=\uvec{u}_h-\uvec{w}_h^\bot$, so that
  \begin{equation}\label{eq:uCh.vh}
    \text{%
      $\uCh{k}\uvec{v}_h\overset{\eqref{eq:uCh.I.u=uCh.w}}=\uCh{k}\uvec{e}_h$ and $\uvec{e}_h=\uvec{v}_h-\uvec{w}_h^\sharp$,
    }
  \end{equation}
  we obtain
  \begin{align*}
    \mathcal E_{h,1}(\bvec{u},p;\uvec{v}_h)
    &=
    \nu\norm{\DIV,h}{\uCh{k}\uvec{e}_h}^2\\
    &\quad
    +\int_\Omega \Big(
    \cCh{k}\uvec{e}_h\times\Pcurl{k}{h}\uvec{u}_h
    \Big)\cdot\Pcurl{k}{h}\uvec{v}_h + \int_\Omega\cancel{\Big(\cCh{k}\Icurl{k}{h}\bvec{u}\times\Pcurl{k}{h}\uvec{v}_h\Big)\cdot\Pcurl{k}{h}\uvec{v}_h}\\
    &\quad
    -\int_\Omega\Big(
    \cCh{k}\Icurl{k}{h}\bvec{u}\times\Pcurl{k}{h}\uvec{w}_h^\sharp
    \Big)\cdot\Pcurl{k}{h}\uvec{v}_h+\cancel{(\uGh{k}\underline{\epsilon}_h,\uvec{v}_h)},
  \end{align*}
  where the first cancellation is due to the fact that $\cCh{k}\Icurl{k}{h}\bvec{u}\times\Pcurl{k}{h}\uvec{v}_h$ is orthogonal to $\Pcurl{k}{h}\uvec{v}_h$, while the second comes from the fact that $\uvec{v}_h\in(\Image\uGh{k})^\bot$ (see \eqref{eq:ddr.incompressibility}).
  Using the generalised H\"older inequality with exponents $(2,4,4)$, we infer
  \begin{equation}\label{eq:stab.1}
    \begin{aligned}
      \nu\norm{\DIV,h}{\uCh{k}\uvec{e}_h}^2\le{}&
      \mathcal E_{h,1}(\bvec{u},p;\uvec{v}_h)+\norm{\bvec{L}^2(\Omega)}{\cCh{k}\uvec{e}_h}\norm{\bvec{L}^4(\Omega)}{\Pcurl{k}{h}\uvec{u}_h}\norm{\bvec{L}^4(\Omega)}{\Pcurl{k}{h}\uvec{v}_h}
      \\
      &+\norm{\bvec{L}^4(\Omega)}{\cCh{k}\Icurl{k}{h}\bvec{u}}\norm{\bvec{L}^2(\Omega)}{\Pcurl{k}{h}\uvec{w}_h^\sharp}\norm{\bvec{L}^4(\Omega)}{\Pcurl{k}{h}\uvec{v}_h}.
    \end{aligned}
  \end{equation}
  Using \eqref{eq:cCh} together with \eqref{def:CcPdiv} and the definition \eqref{def:Csobh} of $\Csobh$ (recall that both $\uvec{u}_h$ and $\uvec{v}_h$ belong to $(\Image\uGh{k})^\perp$), we have
  \[
  \begin{aligned}
    \norm{\bvec{L}^2(\Omega)}{\cCh{k}\uvec{e}_h}\norm{\bvec{L}^4(\Omega)}{\Pcurl{k}{h}\uvec{u}_h}\norm{\bvec{L}^4(\Omega)}{\Pcurl{k}{h}\uvec{v}_h}
    &\le
    \CPdiv\norm{\DIV,h}{\uCh{k}\uvec{e}_h}\Csobh\norm{\DIV,h}{\uCh{k}\uvec{u}_h}
    \Csobh\norm{\DIV,h}{\uCh{k}\uvec{v}_h}
    \\
    \overset{\eqref{eq:uCh.vh}}&=
    \CPdiv\norm{\DIV,h}{\uCh{k}\uvec{e}_h}^2\Csobh^2\norm{\DIV,h}{\uCh{k}\uvec{u}_h}
    \\
    \overset{\eqref{eq:est.u.robust}}&\le
    \CPdiv\Csobh^2\Cpoin\nu^{-1}\norm{\CURL,h}{\Icurl{k}{h} \bvec{R}_{\bvec{u}}}
    \norm{\DIV,h}{\uCh{k}\uvec{e}_h}^2.
  \end{aligned}
  \]
  Notice that the usage of \eqref{eq:est.u.robust} is justified since, as seen in the proof of Lemma \ref{lem:consistency.bounds}, the assumptions \eqref{eq:assum.f} and \eqref{eq:regularity.u} ensure that $p_{|T}\in C^1(\overline{T})$ for all $T\in\Th$.
  Plugging this bound into \eqref{eq:stab.1}, recalling the definition \eqref{eq:def.chi} of $\chi$, and using the definition \eqref{eq:norm.U.h} of $\norm{\bvec{U},h}{{\cdot}}$ along with that of the corresponding dual norm of $\mathcal E_{h,1}(\bvec{u},p;\cdot)$ leads to
  \begin{equation}\label{eq:stab.2}
    \begin{aligned}
      \chi\norm{\DIV,h}{\uCh{k}\uvec{e}_h}^2
      &\le \norm{\bvec{U},h,*}{\mathcal E_{h,1}(\bvec{u},p;\cdot)}
      \left(\norm{\CURL,h}{\uvec{v}_h}^2+\norm{\DIV,h}{\uCh{k}\uvec{v}_h}^2\right)^{\frac12}
      \\
      &\quad
      + \norm{\bvec{L}^4(\Omega)}{\cCh{k}\Icurl{k}{h}\bvec{u}}\norm{\bvec{L}^2(\Omega)}{\Pcurl{k}{h}\uvec{w}_h^\sharp}\norm{\bvec{L}^4(\Omega)}{\Pcurl{k}{h}\uvec{v}_h}
      \\
      &\lesssim
      \norm{\bvec{U},h,*}{\mathcal E_{h,1}(\bvec{u},p;\cdot)}\norm{\DIV,h}{\uCh{k}\uvec{e}_h}
      +\Csobh\seminorm{\bvec{W}^{1,4}(\Omega)}{\bvec{u}}
      \norm{\CURL,h}{\uvec{w}_h^\sharp}\norm{\DIV,h}{\uCh{k}\uvec{e}_h},
    \end{aligned}
  \end{equation}
  where the second inequality follows from \eqref{def:CP} together with $\uCh{k}\uvec{v}_h=\uCh{k}\uvec{e}_h$, \eqref{eq:bound.dI} (with $m=0$) below, \eqref{def:CcPcurl}, and \eqref{def:Csobh}.

  We then use \eqref{eq:def.Eh.p} and recall that $\Icurl{k}{h}\bvec{u}=\uvec{w}_h^\sharp+\uvec{w}_h^\perp$ with $\uvec{w}_h^\perp\in(\Image\uGh{k})^\perp$ to write
  \[
  (\uvec{w}_h^\sharp,\uGh{k}\underline{q}_h)_{\CURL,h}
  = (\Icurl{k}{h}\bvec{u},\uGh{k}\underline{q}_h)_{\CURL,h}\le
  \norm{\bvec{G},h,*}{\mathcal E_{h,2}(\bvec{u},p;\cdot)}\norm{\CURL,h}{\uGh{k}\underline{q}_h}\quad\forall\underline{q}_h\in\Xgrads{k}{h}.
  \]
  Taking the supremum over all $\underline{q}_h\in\Xgrads{k}{h}$ such that $\norm{\CURL,h}{\uGh{k}\underline{q}_h}\le 1$ and recalling that $\uvec{w}_h^\sharp\in\Image\uGh{k}=\uGh{k}(\Xgrads{k}{h})$ leads to 
  \begin{equation}\label{eq:bound.wsharp}
    \norm{\CURL,h}{\uvec{w}_h^\sharp}\le \norm{\bvec{G},h,*}{\mathcal E_{h,2}(\bvec{u},p;\cdot)}.
  \end{equation}
  Plugging this bound into \eqref{eq:stab.2}, recalling the definition \eqref{eq:def.chi} of $\chi$, and simplifying by $\norm{\DIV,h}{\uCh{k}\uvec{e}_h}$ leads to
  \[
  \chi\norm{\DIV,h}{\uCh{k}\uvec{e}_h}
  \lesssim \norm{\bvec{U},h,*}{\mathcal E_{h,1}(\bvec{u},p;\cdot))}
  + \Csobh \seminorm{\bvec{W}^{1,4}(\Omega)}{\bvec{u}}
  \norm{\bvec{G},h,*}{\mathcal E_{h,2}(\bvec{u},p;\cdot)}.
  \]
  Invoking the consistency error bounds \eqref{eq:est.Eh.u} and \eqref{eq:est.Eh.p} concludes the proof of the
  bound on $\uCh{k}\uvec{e}_h=\uCh{k}(\uvec{u}_h-\Icurl{k}{h}\bvec{u})$ stated in \eqref{eq:error.estimate.u}.

  To bound $\uvec{e}_h=\uvec{u}_h-\Icurl{k}{h}\bvec{u}=\uvec{v}_h-\uvec{w}_h^\sharp$ itself, we use \eqref{eq:bound.wsharp} together with the Poincar\'e inequality \eqref{def:CP} on $\uvec{v}_h\in(\Image\uGh{k})^\perp$ and $\uCh{k}\uvec{e}_h=\uCh{k}\uvec{v}_h$ to write
  \[
  \norm{\CURL,h}{\uvec{e}_h}\lesssim
  \norm{\DIV,h}{\uCh{k}\uvec{e}_h}+\norm{\bvec{G},h,*}{\mathcal E_{h,2}(\bvec{u},p;\cdot)}
  \]
  and conclude with \eqref{eq:est.Eh.p} and the bound already established on $\norm{\DIV,h}{\uCh{k}\uvec{e}_h}$.

  \smallskip

  It remains to bound the error $\underline{\epsilon}_h$ on the pressure. Take $\underline{q}_h\in\Xgrads{k}{h}$ and plug $\uvec{v}_h=\uGh{k}\underline{q}_h$ into \eqref{eq:error.bound.E1}. The complex property $\uCh{k}\circ\uGh{k}=\uvec{0}$, generalised H\"older inequalities, \eqref{eq:cCh}, the bounds \eqref{eq:bound.P} on the potential reconstruction and \eqref{eq:bound.dI} (with $m=0$) on the composition of the discrete curl and the interpolator, and the definitions \eqref{def:Csobh} of $\Csobh$ and \eqref{eq:def.Ls.Xcurl} of $\norm{4,\CURL,h}{{\cdot}}$ yield
  \begin{align*}
    (\uGh{k}\underline{\epsilon}_h,\uGh{k}\underline{q}_h)_{\CURL,h}\lesssim{}&
    \left(
    \Csobh\norm{\DIV,h}{\uCh{k}\uvec{e}_h}\norm{\DIV,h}{\uCh{k}\uvec{u}_h}
    + \seminorm{\bvec{W}^{1,4}(\Th)}{\bvec{u}}\norm{\CURL,h}{\uvec{e}_h}
    \right)\norm{4,\CURL,h}{\uGh{k}\underline{q}_h}\\
    &+\mathcal E_{h,1}(\bvec{u},p;\uGh{k}\underline{q}_h).
  \end{align*}
  The estimate \eqref{eq:error.estimate.p} then follows from the definition of $\norm{P,h}{{\cdot}}$, the error bound \eqref{eq:error.estimate.u}, the a priori bound \eqref{eq:est.u.robust}, the consistency estimate \eqref{eq:est.Eh.u}, and the estimate \eqref{eq:est.G.4.2} which, combined with a Cauchy--Schwarz inequality on $\sum_{T\in\Th}$, implies $\norm{\CURL,h}{\uGh{k}\underline{q}_h}\lesssim \norm{4,\CURL,h}{\uGh{k}\underline{q}_h}$.

  
\section{Essential boundary conditions}\label{sec:essential.BC}

The $\CURL\CURL$ formulation \eqref{eq:momentum.incompressibility} of the Navier--Stokes equations can also be considered with the following essential boundary conditions:
\begin{equation}\label{eq:essential.BCs}
\bvec{u}\times\normal = \bvec{0} \quad\mbox{ and }\quad p=0\quad\mbox{ on $\partial\Omega$}.
\end{equation}
Due to the boundary condition on the pressure, fixing its integral is not required here to ensure its uniqueness.
The spaces embedding these boundary conditions are 
\[
\Hcurlz{\Omega}=\left\{
\bvec{v}\in\Hcurl{\Omega}\,:\,\bvec{v}\times\normal=\bvec{0}\mbox{ on $\partial\Omega$}
\right\}\,,
\quad
H^1_0(\Omega)=\left\{
q\in H^1(\Omega)\,:\,q=0\mbox{ on $\partial\Omega$}
\right\}.
\]
The traces appearing here are well-defined (weakly for functions in $\Hcurl{\Omega}$, strongly for functions in $H^1(\Omega)$), see \cite{Assous.Ciarlet.ea:18}. The weak formulation of problem \eqref{eq:momentum.incompressibility} completed with the essential boundary conditions \eqref{eq:essential.BCs} is then:
\begin{equation}\label{eq:weak.essential}
  \begin{aligned}
    &\text{Find $(\bvec{u},p)\in \Hcurlz{\Omega}\times H^1_0(\Omega)$ such that, for all $(\bvec{v},q)\in \Hcurlz{\Omega}\times H^1_0(\Omega)$,}
    \\
    &\nu\int_\Omega\CURL\bvec{u}\cdot\CURL\bvec{v} + \int_\Omega [\CURL \bvec{u}\times \bvec{u}]\cdot\bvec{v}+\int_\Omega\GRAD p\cdot\bvec{v} = \int_\Omega\bvec{f}\cdot\bvec{v},
    \\
    &-\int_\Omega\bvec{u}\cdot\GRAD q = 0.
  \end{aligned}
\end{equation}

The SDDR spaces with essential boundary conditions, which replace $H^1_0(\Omega)$ and $\Hcurlz{\Omega}$ at the discrete level, are
\begin{align*}
  \Xgradz{k}{h}\coloneq{}&
  \big\{
    \begin{aligned}[t]
      \underline{q}_h\in\Xgrad{k}{h}\,\st\,
      q_F=0\quad\forall F\in\Fhb\,,\quad
      q_E=0\quad\forall E\in\Ehb\,,\quad q_V=0\quad\forall V\in\Vhb
    \big\},
    \end{aligned}\\
  \Xcurlz{k}{h}\coloneq{}&
  \big\{
  \begin{aligned}[t]
    \uvec{v}_h\in\Xcurl{k}{h}\,\st\,\bvec{v}_{\cvec{R},F}=\bvec{v}_{\cvec{R},F}^\compl=\bvec{0}\quad\forall F\in\Fhb\,,\quad
    v_E=0\quad\forall E\in\Ehb
    \big\},
  \end{aligned}
\end{align*}
where $\Fhb$ is the set of boundary faces, $\Ehb$ is the set of boundary edges, and $\Vhb$ is the set of boundary vertices. The SDDR scheme for \eqref{eq:weak.essential} is: Find $(\uvec{u}_h,\underline{p}_h)\in \Xcurlz{k}{h}\times \Xgradz{k}{h}$ such that \eqref{eq:ddr} holds for all $(\uvec{v}_h,\underline{q}_h)\in \Xcurlz{k}{h}\times \Xgradz{k}{h}$.

We can easily check that, for any $\underline{q}_h\in\Xgradz{k}{h}$, $\uGh{k}\underline{q}_h\in\Xcurlz{k}{h}$. As a consequence, $\uvec{v}_h=\uGh{k}\underline{p}_h$ is a valid test function in the scheme and its analysis can be carried
out as the analysis of the scheme \eqref{eq:ddr} for the natural boundary conditions, re-defining $\Csob$ by replacing in \eqref{def:Csobh} the spaces $\Xcurl{k}{h},\Xgrads{k}{h}$ by $\Xcurlz{k}{h},\Xgradz{k}{h}$.

\begin{remark}[Mixed boundary conditions]
We can also consider mixed boundary conditions, imposing natural boundary conditions (that is, $\CURL\bvec{u}\times\normal$ and $\bvec{u}\cdot\normal$) on a subset $\Gamma_D$ of $\partial\Omega$ and essential boundary conditions (that is, $\bvec{u}\times\normal$ and $p$) on the rest of $\partial \Omega$.
The exactness of the de Rham (and thus the DDR) complex for those -- and thus the well-posedness of the continuous and discrete models --, however, depends on the topology of $\Gamma_D$ \cite{Licht:19,Goldshtein.Mitrea.ea:11}.
\end{remark}


\section{Numerical tests}\label{sec:num.tests}

The SDDR scheme \eqref{eq:ddr} has been implemented in the open source C++ library \texttt{HArDCore3D} (see \url{https://github.com/jdroniou/HArDCore}). Static condensation was used to reduce the size of the globally coupled system, which was then solved using the \texttt{Intel MKL PARDISO} library (see \url{https://software.intel.com/en-us/mkl})~\cite{Alappat.Basermann.ea:20}. 

\subsection{Analytical solution}\label{sec:analytical.solution}

We first test the SDDR scheme \eqref{eq:ddr} by selecting the same analytical solution as in \cite{Beirao-da-Veiga.Dassi.ea:22}, namely
\[
p(x,y,z)=\lambda\sin(2\pi x)\sin(2\pi y)\sin(2\pi z)\quad\mbox{ and }\quad
\bvec{u}(x,y,z)=\begin{bmatrix}
  \frac12 \sin(2\pi x)\cos(2\pi y)\cos(2\pi z)\\[.5em]
  \frac12 \cos(2\pi x)\sin(2\pi y)\cos(2\pi z)\\[.5em]
  -\cos(2\pi x)\cos(2\pi y)\sin(2\pi z)
  \end{bmatrix},
\]
with Reynolds number $1$ and forcing term $\bvec{f}$ selected accordingly.
Here, $\lambda$ is a parameter that we can select to demonstrate the robustness
of the scheme with respect to the pressure magnitude. We run the scheme on two families of meshes of $\Omega\coloneq (0,1)^3$: tetrahedral and Voronoi (see \cite[Fig.~1]{Beirao-da-Veiga.Dassi.ea:22}), and measure two types of errors on the velocity and pressure:
the discrete errors (corresponding to the left-hand sides in the error estimates \eqref{eq:error.estimate.u} and \eqref{eq:error.estimate.p})
\[
E^{\rm d}_{\bvec{u}}\coloneq \norm{\bvec{U},h}{\uvec{u}_h-\Icurl{k}{h}\bvec{u}}
\,,\quad
E^{\rm d}_{p}\coloneq\norm{\CURL,h}{\uGh{k}(\underline{p}_h-\Igrad{k}{h}p)},
\]
and the potential-based errors
\[
E^{\rm p}_{\bvec{u}}\coloneq\Big(\norm{\bvec{L}^2(\Omega)}{\Pcurl{k}{h}\uvec{u}_h-\bvec{u}}^2+\norm{\bvec{L}^2(\Omega)}{\Pdiv{k}{h}\uCh{k}\uvec{u}_h-\CURL\bvec{u}}^2\Big)^{\frac12}
\,,\quad
E^{\rm p}_{p}\coloneq\norm{\bvec{L}^2(\Omega)}{\Pcurl{k}{h}\uGh{k}\underline{p}_h-\GRAD p},
\]
where we remind the reader that $\Pcurl{k}{h}$ and $\Pdiv{k}{h}$ are obtained patching the corresponding local potentials.

The loglog graphs of the errors vs.~$h$ are presented in Figures \ref{fig:conv.tetra} and \ref{fig:conv.voro}.
The errors estimates \eqref{eq:error.estimate.u} and \eqref{eq:error.estimate.p} are independent of the pressure magnitude, and this is reflected in 
these figures by the fact that the discrete errors on the velocity and the pressure are unaffected by a large increase in the scaling factor
$\lambda$. Using \cite[Theorem 6]{Di-Pietro.Droniou:23}, it can be checked that the continuous errors for the velocity and pressure, respectively, are 
bounded above by the sum of the respective discrete error and an approximation error term depending on the derivative of the corresponding unknowns.
The graphs also reflect this: the potential-based error for the velocity remains unaffected by a change of magnitude of the pressure,
while that on the pressure is, as expected, degraded by an increase of $\lambda$ (but the relative errors remains similar for both
values of $\lambda$).
Overall, the rates of convergence in these tests follow the prediction of Theorem \ref{thm:convergence}, except for the error on $\nabla p$ for $k=0$ which is a bit below 1 (this could be due to the dependence of the error on the mesh regularity factor --which is not uniform on these meshes-- or to the asymptotic rate not being reached yet for this low-order approximation).

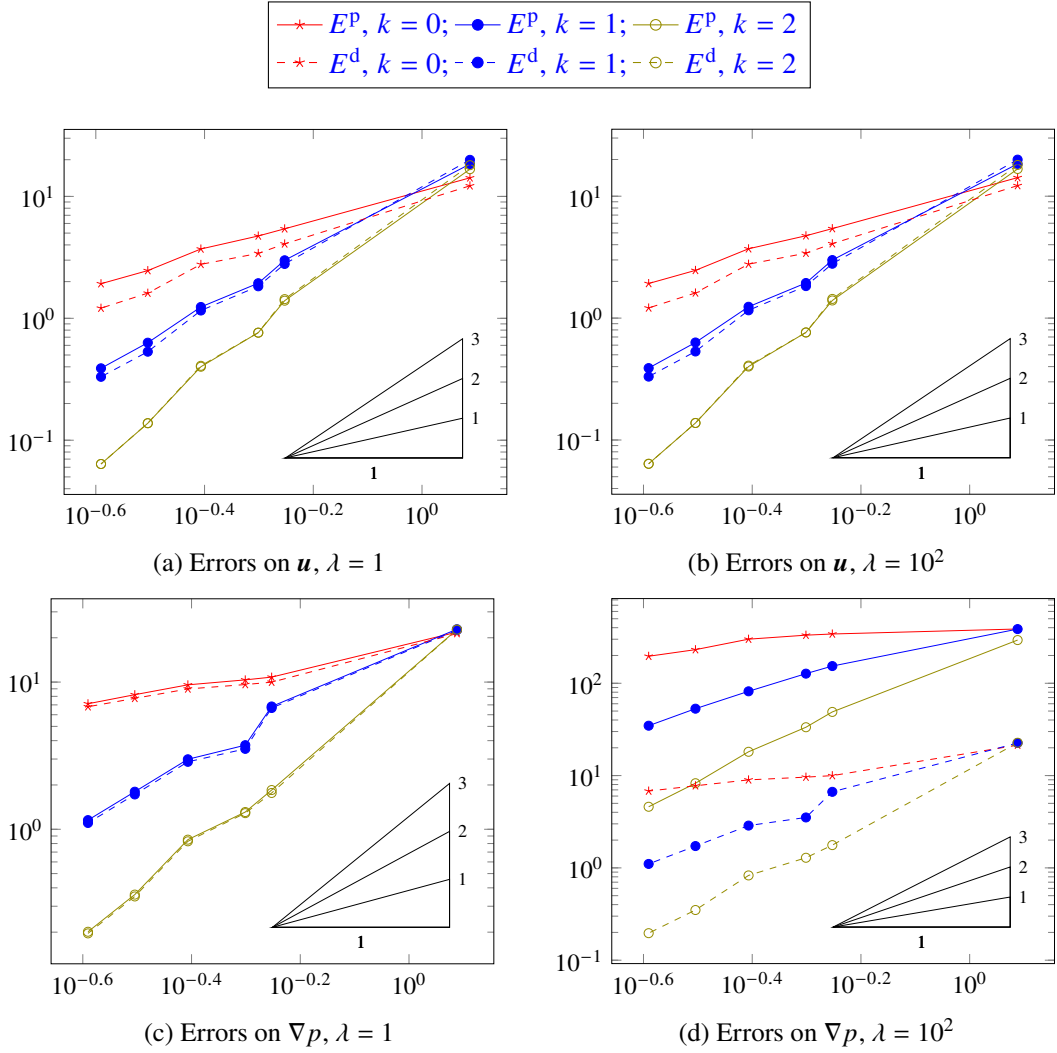
\begin{figure}\centering
  \ref{ddr.conv.tets}
  \vspace{0.50cm}\\
  \begin{minipage}{0.45\textwidth}
    \begin{tikzpicture}[scale=0.85]
      \begin{loglogaxis} [legend columns=3, legend to name=ddr.conv.tets]  
        \logLogSlopeTriangle{0.90}{0.4}{0.1}{1}{black};
        \addplot [mark=star, red] table[x=MeshSize,y=absE_cHcurlVel] {outputs/trigo/Tetgen-Cube-0_k0_Re1_Psca1/data_rates.dat};
        \addlegendentry{$E^{\rm p}$, $k=0$;}
        \logLogSlopeTriangle{0.90}{0.4}{0.1}{2}{black};
        \addplot [mark=*, blue] table[x=MeshSize,y=absE_cHcurlVel] {outputs/trigo/Tetgen-Cube-0_k1_Re1_Psca1/data_rates.dat};
        \addlegendentry{$E^{\rm p}$, $k=1$;}
        \logLogSlopeTriangle{0.90}{0.4}{0.1}{3}{black};
        \addplot [mark=o, olive] table[x=MeshSize,y=absE_cHcurlVel] {outputs/trigo/Tetgen-Cube-0_k2_Re1_Psca1/data_rates.dat};
        \addlegendentry{$E^{\rm p}$, $k=2$}
        \addplot [mark=star, mark options=solid, red, dashed] table[x=MeshSize,y=absE_HcurlVel] {outputs/trigo/Tetgen-Cube-0_k0_Re1_Psca1/data_rates.dat};
        \addlegendentry{$E^{\rm d}$, $k=0$;}
        \addplot [mark=*, mark options=solid, blue, dashed] table[x=MeshSize,y=absE_HcurlVel] {outputs/trigo/Tetgen-Cube-0_k1_Re1_Psca1/data_rates.dat};
        \addlegendentry{$E^{\rm d}$, $k=1$;}
        \addplot [mark=o, mark options=solid, olive, dashed] table[x=MeshSize,y=absE_HcurlVel] {outputs/trigo/Tetgen-Cube-0_k2_Re1_Psca1/data_rates.dat};
        \addlegendentry{$E^{\rm d}$, $k=2$}
      \end{loglogaxis}            
    \end{tikzpicture}
    \subcaption{Errors on $\bvec{u}$, $\lambda=1$}
  \end{minipage}
  \begin{minipage}{0.45\textwidth}
    \begin{tikzpicture}[scale=0.85]
      \begin{loglogaxis} 
        \logLogSlopeTriangle{0.90}{0.4}{0.1}{1}{black};
        \addplot [mark=star, red] table[x=MeshSize,y=absE_cHcurlVel] {outputs/trigo/Tetgen-Cube-0_k0_Re1_Psca1e2/data_rates.dat};
        \logLogSlopeTriangle{0.90}{0.4}{0.1}{2}{black};
        \addplot [mark=*, blue] table[x=MeshSize,y=absE_cHcurlVel] {outputs/trigo/Tetgen-Cube-0_k1_Re1_Psca1e2/data_rates.dat};
        \logLogSlopeTriangle{0.90}{0.4}{0.1}{3}{black};
        \addplot [mark=o, olive] table[x=MeshSize,y=absE_cHcurlVel] {outputs/trigo/Tetgen-Cube-0_k2_Re1_Psca1e2/data_rates.dat};
        \addplot [mark=star, mark options=solid, red, dashed] table[x=MeshSize,y=absE_HcurlVel] {outputs/trigo/Tetgen-Cube-0_k0_Re1_Psca1e2/data_rates.dat};
        \addplot [mark=*, mark options=solid, blue, dashed] table[x=MeshSize,y=absE_HcurlVel] {outputs/trigo/Tetgen-Cube-0_k1_Re1_Psca1e2/data_rates.dat};
        \addplot [mark=o, mark options=solid, olive, dashed] table[x=MeshSize,y=absE_HcurlVel] {outputs/trigo/Tetgen-Cube-0_k2_Re1_Psca1e2/data_rates.dat};
      \end{loglogaxis}            
    \end{tikzpicture}
    \subcaption{Errors on $\bvec{u}$, $\lambda=10^2$}
  \end{minipage}
  \\[0.5em]
  \begin{minipage}{0.45\textwidth}
    \begin{tikzpicture}[scale=0.85]
      \begin{loglogaxis} 
        \logLogSlopeTriangle{0.90}{0.4}{0.1}{1}{black};
        \addplot [mark=star, red] table[x=MeshSize,y=absE_cL2GradPre] {outputs/trigo/Tetgen-Cube-0_k0_Re1_Psca1/data_rates.dat};
        \logLogSlopeTriangle{0.90}{0.4}{0.1}{2}{black};
        \addplot [mark=*, blue] table[x=MeshSize,y=absE_cL2GradPre] {outputs/trigo/Tetgen-Cube-0_k1_Re1_Psca1/data_rates.dat};
        \logLogSlopeTriangle{0.90}{0.4}{0.1}{3}{black};
        \addplot [mark=o, olive] table[x=MeshSize,y=absE_cL2GradPre] {outputs/trigo/Tetgen-Cube-0_k2_Re1_Psca1/data_rates.dat};
        \addplot [mark=star, mark options=solid, red, dashed] table[x=MeshSize,y=absE_L2GradPre] {outputs/trigo/Tetgen-Cube-0_k0_Re1_Psca1/data_rates.dat};
        \addplot [mark=*, mark options=solid, blue, dashed] table[x=MeshSize,y=absE_L2GradPre] {outputs/trigo/Tetgen-Cube-0_k1_Re1_Psca1/data_rates.dat};
        \addplot [mark=o, mark options=solid, olive, dashed] table[x=MeshSize,y=absE_L2GradPre] {outputs/trigo/Tetgen-Cube-0_k2_Re1_Psca1/data_rates.dat};
      \end{loglogaxis}            
    \end{tikzpicture}
    \subcaption{Errors on $\nabla p$, $\lambda=1$}
  \end{minipage}
  \begin{minipage}{0.45\textwidth}
    \begin{tikzpicture}[scale=0.85]
      \begin{loglogaxis} 
        \logLogSlopeTriangle{0.90}{0.4}{0.1}{1}{black};
        \addplot [mark=star, red] table[x=MeshSize,y=absE_cL2GradPre] {outputs/trigo/Tetgen-Cube-0_k0_Re1_Psca1e2/data_rates.dat};
        \logLogSlopeTriangle{0.90}{0.4}{0.1}{2}{black};
        \addplot [mark=*, blue] table[x=MeshSize,y=absE_cL2GradPre] {outputs/trigo/Tetgen-Cube-0_k1_Re1_Psca1e2/data_rates.dat};
        \logLogSlopeTriangle{0.90}{0.4}{0.1}{3}{black};
        \addplot [mark=o, olive] table[x=MeshSize,y=absE_cL2GradPre] {outputs/trigo/Tetgen-Cube-0_k2_Re1_Psca1e2/data_rates.dat};
        \addplot [mark=star, mark options=solid, red, dashed] table[x=MeshSize,y=absE_L2GradPre] {outputs/trigo/Tetgen-Cube-0_k0_Re1_Psca1e2/data_rates.dat};
        \addplot [mark=*, mark options=solid, blue, dashed] table[x=MeshSize,y=absE_L2GradPre] {outputs/trigo/Tetgen-Cube-0_k1_Re1_Psca1e2/data_rates.dat};
        \addplot [mark=o, mark options=solid, olive, dashed] table[x=MeshSize,y=absE_L2GradPre] {outputs/trigo/Tetgen-Cube-0_k2_Re1_Psca1e2/data_rates.dat};
      \end{loglogaxis}            
    \end{tikzpicture}
    \subcaption{Errors on $\nabla p$, $\lambda=10^2$}
  \end{minipage}
  \caption{Analytical test of Section \ref{sec:analytical.solution}, Tetrahedral meshes: errors with respect to  $h$}
  \label{fig:conv.tetra}
\end{figure}

\begin{figure}\centering
  \ref{ddr.conv.voro}
  \vspace{0.50cm}\\
  \begin{minipage}{0.45\textwidth}
    \begin{tikzpicture}[scale=0.85]
      \begin{loglogaxis} [legend columns=3, legend to name=ddr.conv.voro]  
        \logLogSlopeTriangle{0.90}{0.4}{0.1}{1}{black};
        \addplot [mark=star, red] table[x=MeshSize,y=absE_cHcurlVel] {outputs/trigo/Voro-small-0_k0_Re1_Psca1/data_rates.dat};
        \addlegendentry{$E^{\rm p}$, $k=0$;}
        \logLogSlopeTriangle{0.90}{0.4}{0.1}{2}{black};
        \addplot [mark=*, blue] table[x=MeshSize,y=absE_cHcurlVel] {outputs/trigo/Voro-small-0_k1_Re1_Psca1/data_rates.dat};
        \addlegendentry{$E^{\rm p}$, $k=1$;}
        \logLogSlopeTriangle{0.90}{0.4}{0.1}{3}{black};
        \addplot [mark=o, olive] table[x=MeshSize,y=absE_cHcurlVel] {outputs/trigo/Voro-small-0_k2_Re1_Psca1/data_rates.dat};
        \addlegendentry{$E^{\rm p}$, $k=2$}
        \addplot [mark=star, mark options=solid, red, dashed] table[x=MeshSize,y=absE_HcurlVel] {outputs/trigo/Voro-small-0_k0_Re1_Psca1/data_rates.dat};
        \addlegendentry{$E^{\rm d}$, $k=0$;}
        \addplot [mark=*, mark options=solid, blue, dashed] table[x=MeshSize,y=absE_HcurlVel] {outputs/trigo/Voro-small-0_k1_Re1_Psca1/data_rates.dat};
        \addlegendentry{$E^{\rm d}$, $k=1$;}
        \addplot [mark=o, mark options=solid, olive, dashed] table[x=MeshSize,y=absE_HcurlVel] {outputs/trigo/Voro-small-0_k2_Re1_Psca1/data_rates.dat};
        \addlegendentry{$E^{\rm d}$, $k=2$}
      \end{loglogaxis}            
    \end{tikzpicture}
    \subcaption{Errors on $\bvec{u}$, $\lambda=1$}
  \end{minipage}
  \begin{minipage}{0.45\textwidth}
    \begin{tikzpicture}[scale=0.85]
      \begin{loglogaxis} 
        \logLogSlopeTriangle{0.90}{0.4}{0.1}{1}{black};
        \addplot [mark=star, red] table[x=MeshSize,y=absE_cHcurlVel] {outputs/trigo/Voro-small-0_k0_Re1_Psca1e2/data_rates.dat};
        \logLogSlopeTriangle{0.90}{0.4}{0.1}{2}{black};
        \addplot [mark=*, blue] table[x=MeshSize,y=absE_cHcurlVel] {outputs/trigo/Voro-small-0_k1_Re1_Psca1e2/data_rates.dat};
        \logLogSlopeTriangle{0.90}{0.4}{0.1}{3}{black};
        \addplot [mark=o, olive] table[x=MeshSize,y=absE_cHcurlVel] {outputs/trigo/Voro-small-0_k2_Re1_Psca1e2/data_rates.dat};
        \addplot [mark=star, mark options=solid, red, dashed] table[x=MeshSize,y=absE_HcurlVel] {outputs/trigo/Voro-small-0_k0_Re1_Psca1e2/data_rates.dat};
        \addplot [mark=*, mark options=solid, blue, dashed] table[x=MeshSize,y=absE_HcurlVel] {outputs/trigo/Voro-small-0_k1_Re1_Psca1e2/data_rates.dat};
        \addplot [mark=o, mark options=solid, olive, dashed] table[x=MeshSize,y=absE_HcurlVel] {outputs/trigo/Voro-small-0_k2_Re1_Psca1e2/data_rates.dat};
      \end{loglogaxis}            
    \end{tikzpicture}
    \subcaption{Errors on $\bvec{u}$, $\lambda=10^2$}
  \end{minipage}
  \\[0.5em]
  \begin{minipage}{0.45\textwidth}
    \begin{tikzpicture}[scale=0.85]
      \begin{loglogaxis} 
        \logLogSlopeTriangle{0.90}{0.4}{0.1}{1}{black};
        \addplot [mark=star, red] table[x=MeshSize,y=absE_cL2GradPre] {outputs/trigo/Voro-small-0_k0_Re1_Psca1/data_rates.dat};
        \logLogSlopeTriangle{0.90}{0.4}{0.1}{2}{black};
        \addplot [mark=*, blue] table[x=MeshSize,y=absE_cL2GradPre] {outputs/trigo/Voro-small-0_k1_Re1_Psca1/data_rates.dat};
        \logLogSlopeTriangle{0.90}{0.4}{0.1}{3}{black};
        \addplot [mark=o, olive] table[x=MeshSize,y=absE_cL2GradPre] {outputs/trigo/Voro-small-0_k2_Re1_Psca1/data_rates.dat};
        \addplot [mark=star, mark options=solid, red, dashed] table[x=MeshSize,y=absE_L2GradPre] {outputs/trigo/Voro-small-0_k0_Re1_Psca1/data_rates.dat};
        \addplot [mark=*, mark options=solid, blue, dashed] table[x=MeshSize,y=absE_L2GradPre] {outputs/trigo/Voro-small-0_k1_Re1_Psca1/data_rates.dat};
        \addplot [mark=o, mark options=solid, olive, dashed] table[x=MeshSize,y=absE_L2GradPre] {outputs/trigo/Voro-small-0_k2_Re1_Psca1/data_rates.dat};
      \end{loglogaxis}            
    \end{tikzpicture}
    \subcaption{Errors on $\nabla p$, $\lambda=1$}
  \end{minipage}
  \begin{minipage}{0.45\textwidth}
    \begin{tikzpicture}[scale=0.85]
      \begin{loglogaxis} 
        \logLogSlopeTriangle{0.90}{0.4}{0.1}{1}{black};
        \addplot [mark=star, red] table[x=MeshSize,y=absE_cL2GradPre] {outputs/trigo/Voro-small-0_k0_Re1_Psca1e2/data_rates.dat};
        \logLogSlopeTriangle{0.90}{0.4}{0.1}{2}{black};
        \addplot [mark=*, blue] table[x=MeshSize,y=absE_cL2GradPre] {outputs/trigo/Voro-small-0_k1_Re1_Psca1e2/data_rates.dat};
        \logLogSlopeTriangle{0.90}{0.4}{0.1}{3}{black};
        \addplot [mark=o, olive] table[x=MeshSize,y=absE_cL2GradPre] {outputs/trigo/Voro-small-0_k2_Re1_Psca1e2/data_rates.dat};
        \addplot [mark=star, mark options=solid, red, dashed] table[x=MeshSize,y=absE_L2GradPre] {outputs/trigo/Voro-small-0_k0_Re1_Psca1e2/data_rates.dat};
        \addplot [mark=*, mark options=solid, blue, dashed] table[x=MeshSize,y=absE_L2GradPre] {outputs/trigo/Voro-small-0_k1_Re1_Psca1e2/data_rates.dat};
        \addplot [mark=o, mark options=solid, olive, dashed] table[x=MeshSize,y=absE_L2GradPre] {outputs/trigo/Voro-small-0_k2_Re1_Psca1e2/data_rates.dat};
      \end{loglogaxis}            
    \end{tikzpicture}
    \subcaption{Errors on $\nabla p$, $\lambda=10^2$}
  \end{minipage}
  \caption{Analytical test of Section \ref{sec:analytical.solution}, Voronoi meshes: errors with respect to  $h$}
  \label{fig:conv.voro}
\end{figure}
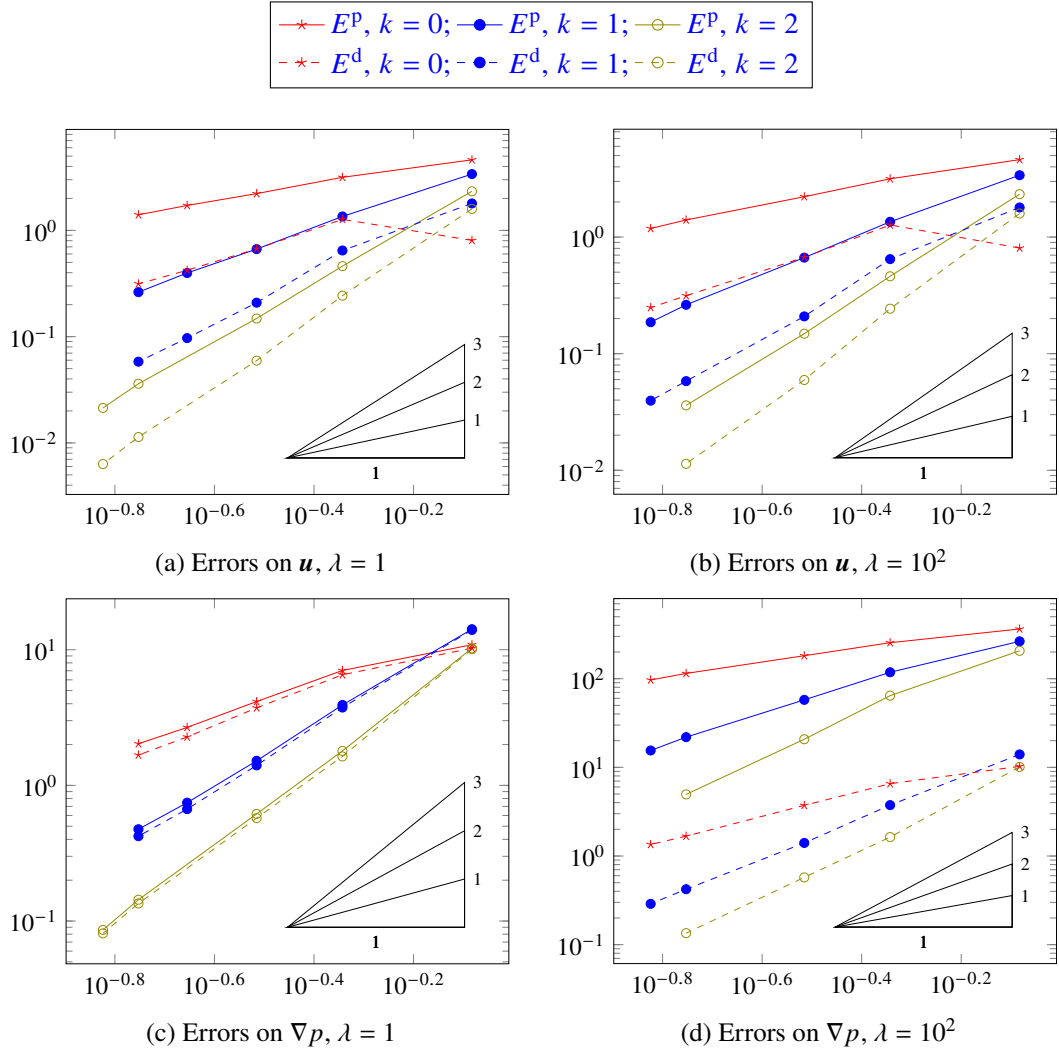

\subsection{Pressure--flux boundary conditions}\label{sec:pressflux}

The series of tests in this section is not based on a known analytical solution. Instead, still considering the unit cube $\Omega$ as domain, we impose mixed boundary conditions, enforcing the pressure on part of one side of the cube and the flux on part of the opposite side. Specifically, we take
\begin{itemize}
\item Essential boundary conditions $p(x,y,z)=-z$ and $\bvec{u}\times\normal=\bvec{0}$ on the bottom corner $\{0\}\times(0,0.25)\times(0,0.25)$ of the face $x=0$ of $\Omega$,
\item Natural boundary conditions $\CURL\bvec{u}\times \normal =\bvec{0}$ and $\bvec{u}\cdot\normal=1$ on the bottom corner $\{1\}\times(0,0.25)\times(0,0.25)$ of the face $x=1$ of $\Omega$,
\item Homogeneous natural boundary conditions $\CURL\bvec{u}\times \normal =\bvec{0}$ and $\bvec{u}\cdot\normal=0$ on the rest of the domain.
\end{itemize}
The essential boundary conditions that are imposed here are fully compatible with the spaces $H^1(\Omega)$ and $\Hcurl{\Omega}$ for the pressure and the velocity, and are therefore valid in a mixed boundary conditions version of \eqref{eq:weak} (since the natural boundary conditions do not require any regularity property on the spaces). We notice, however, that the boundary conditions on $\bvec{u}\cdot\normal$ (which are discontinuous along the face $x=1$) prevent $\bvec{u}$ from being in $\bvec{H}^1(\Omega)$, and that a weak formulation of the Laplacian-based model \eqref{eq:strong.Delta} would not allow us to impose such flux boundary conditions on the velocity (even disregarding the boundary conditions on the vorticity).

The meshes for these tests are Cartesian meshes made of $n^3$ cubes with $n\in\{4,8,16,32,48,64\}$, and we consider the SDDR degrees $k=0,1,2$; the finest mesh/highest degrees $(n,k)=(64,1)$, $(n,k)=(48,2)$ and $(n,k)=(64,2)$ are not shown due to the limitations of our (direct) solver. The tests are run with a Reynolds number of 100. The velocity streamlines and pressure obtained for $n=32$ and $k=2$ are presented in Figure \ref{fig:pressflux_picture}.

\begin{figure}\centering
  \begin{minipage}{0.45\textwidth}
    \includegraphics[width=\textwidth]{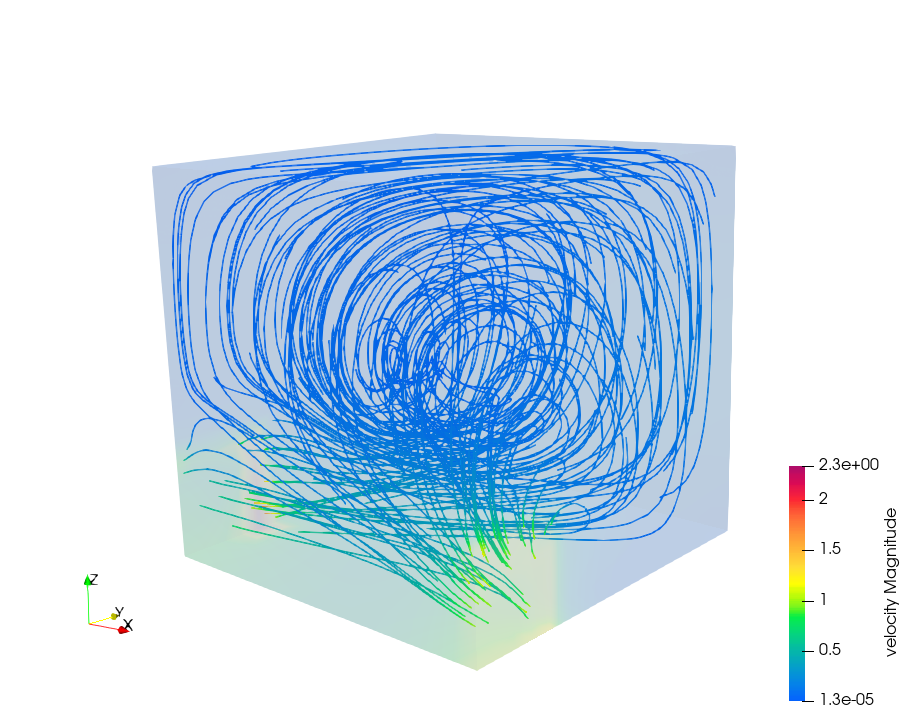}
    \subcaption{Velocity}
  \end{minipage}
  \hspace{0.25cm}
  \begin{minipage}{0.45\textwidth}
    \includegraphics[width=\textwidth]{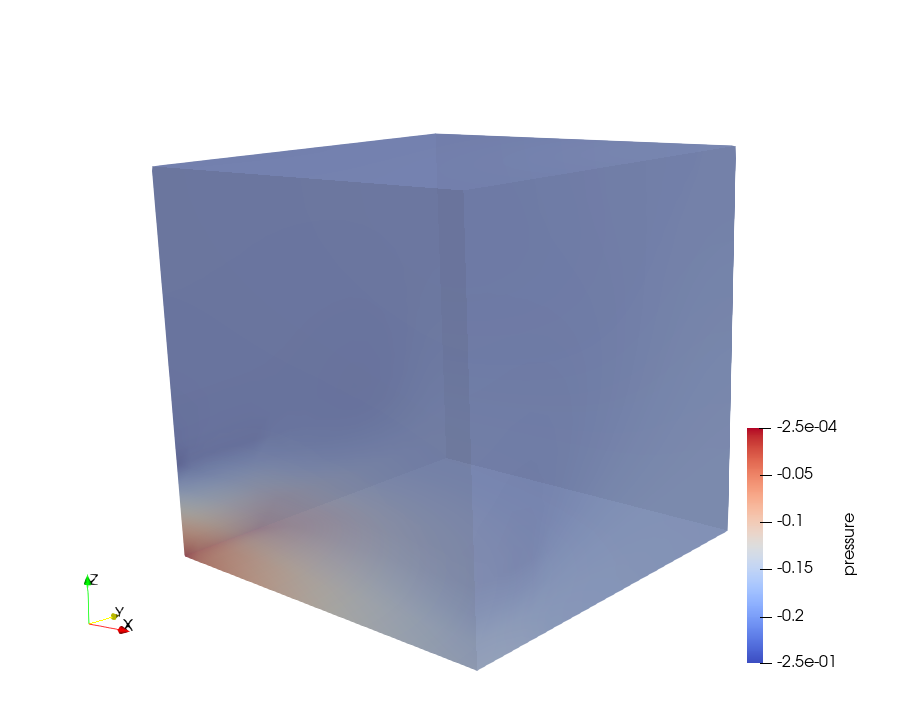}
    \subcaption{Pressure}
  \end{minipage}
  \caption{Velocity streamlines and pressure for the test in Section \ref{sec:pressflux}}
  \label{fig:pressflux_picture}
\end{figure}

Figure \ref{fig:pressflux_norms} presents the convergence, as $h$ is refined, of the discrete norms of the velocity and the pressure. In Figure \ref{fig:pressflux_convergence} we show the error between these norms, and reference values obtained with $n=32$ and $k=2$. We notice that the scheme for $k=0,1$ achieves the expected rate of convergence $k+1$, but that $k=2$ only provides a minor improvement over $k=1$, and certainly not a gain in terms of convergence rates. This is probably due to the lack of regularity, discussed above, of the exact velocity, which limits the benefit of increasing the polynomial degree beyond $k=1$. There is, however, a real gain, of about an order of magnitude, between $k=0$ and $k=1$, which indicates that even when the solution is not expected to be very smooth, increasing slightly the polynomial degree of the approximation can result in a real benefit in terms of accuracy vs.~cost (a similar conclusion was reached, for a different model and scheme, in \cite{Anderson.Droniou:18}).

\begin{figure}\centering
  \ref{ddr.pressflux.2}
  \vspace{0.50cm}\\
  \begin{minipage}{0.45\textwidth}
    \begin{tikzpicture}[scale=0.85]
      \begin{loglogaxis} [legend columns=3, legend to name=ddr.pressflux.2]  
        \addplot [mark=star, red] table[x=MeshSize,y=discN_HCurlVel] {outputs/pressflux/Cubic-Cells_k0_Re100_Psca1/norms.dat};
        \addlegendentry{$k=0$;}
        \addplot [mark=*, blue] table[x=MeshSize,y=discN_HCurlVel] {outputs/pressflux/Cubic-Cells_k1_Re100_Psca1/norms.dat};
        \addlegendentry{$k=1$;}
        \addplot [mark=o, olive] table[x=MeshSize,y=discN_HCurlVel] {outputs/pressflux/Cubic-Cells_k2_Re100_Psca1/norms.dat};
        \addlegendentry{$k=2$}
      \end{loglogaxis}            
    \end{tikzpicture}
    \subcaption{Discrete $\Hcurl{\Omega}$ norm on $\bvec{u}$}
  \end{minipage}
  \begin{minipage}{0.45\textwidth}
    \begin{tikzpicture}[scale=0.85]
      \begin{loglogaxis} 
        \addplot [mark=star, red] table[x=MeshSize,y=discN_HGradPre] {outputs/pressflux/Cubic-Cells_k0_Re100_Psca1/norms.dat};
        \addplot [mark=*, blue] table[x=MeshSize,y=discN_HGradPre] {outputs/pressflux/Cubic-Cells_k1_Re100_Psca1/norms.dat};
        \addplot [mark=o, olive] table[x=MeshSize,y=discN_HGradPre] {outputs/pressflux/Cubic-Cells_k2_Re100_Psca1/norms.dat};
      \end{loglogaxis}            
    \end{tikzpicture}
    \subcaption{Discrete $H^1(\Omega)$ norm on $p$}
  \end{minipage}
  \caption{Pressure--flux test of Section \ref{sec:pressflux}, discrete norms w.r.t.~$h$}
  \label{fig:pressflux_norms}
\end{figure}
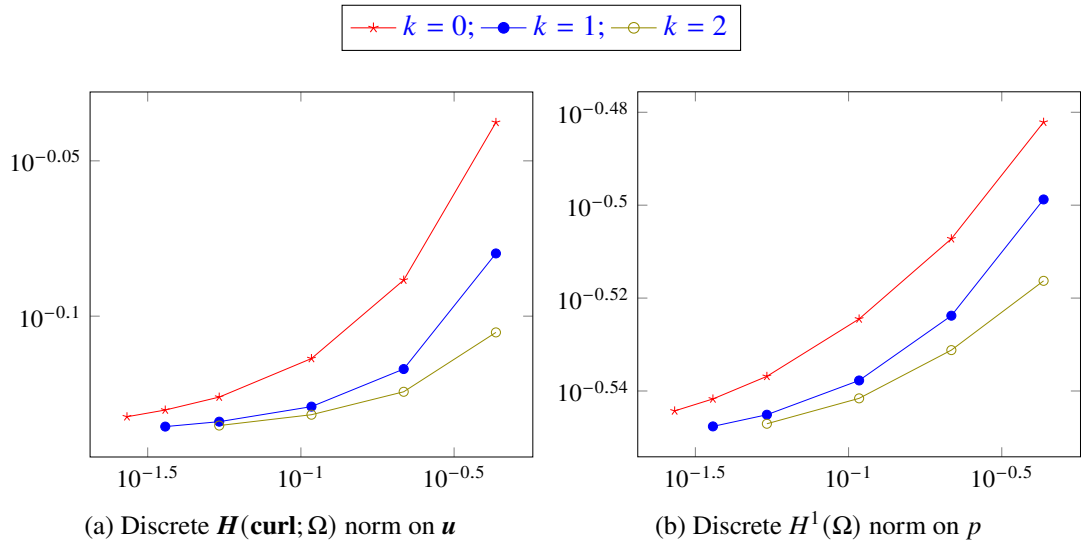

\def\refnormHCurlPscaOne{7.3256611669273153e-01}
\def\refnormHGradPscaOne{2.8368266709481171e-01}
\begin{figure}\centering
  \ref{ddr.pressflux}
  \vspace{0.50cm}\\
  \begin{minipage}{0.45\textwidth}
    \begin{tikzpicture}[scale=0.85]
      \begin{loglogaxis} [legend columns=3, legend to name=ddr.pressflux]  
        \logLogSlopeTriangle{0.90}{0.3}{0.1}{1}{black};
        \addplot [mark=star, red] table[x=MeshSize,y expr=abs(\thisrow{discN_HCurlVel}-\refnormHCurlPscaOne)] {outputs/pressflux/Cubic-Cells_k0_Re100_Psca1/norms.dat};
        \addlegendentry{$k=0$;}
        \logLogSlopeTriangle{0.90}{0.3}{0.1}{2}{black};
        \addplot [mark=*, blue] table[x=MeshSize,y expr=abs(\thisrow{discN_HCurlVel}-\refnormHCurlPscaOne)] {outputs/pressflux/Cubic-Cells_k1_Re100_Psca1/norms.dat};
        \addlegendentry{$k=1$;}
        \logLogSlopeTriangle{0.90}{0.3}{0.1}{3}{black};
        \addplot [mark=o, olive] table[x=MeshSize,y expr=abs(\thisrow{discN_HCurlVel}-\refnormHCurlPscaOne)] {outputs/pressflux/Cubic-Cells_k2_Re100_Psca1/norms.dat};
        \addlegendentry{$k=2$}
      \end{loglogaxis}            
    \end{tikzpicture}
    \subcaption{Errors on discrete $\Hcurl{\Omega}$ norm on $\bvec{u}$}
  \end{minipage}
  \begin{minipage}{0.45\textwidth}
    \begin{tikzpicture}[scale=0.85]
      \begin{loglogaxis} 
        \logLogSlopeTriangle{0.90}{0.3}{0.1}{1}{black};
        \addplot [mark=star, red] table[x=MeshSize,y expr=abs(\thisrow{discN_HGradPre}-\refnormHGradPscaOne)] {outputs/pressflux/Cubic-Cells_k0_Re100_Psca1/norms.dat};
        \logLogSlopeTriangle{0.90}{0.3}{0.1}{2}{black};
        \addplot [mark=*, blue] table[x=MeshSize,y expr=abs(\thisrow{discN_HGradPre}-\refnormHGradPscaOne)] {outputs/pressflux/Cubic-Cells_k1_Re100_Psca1/norms.dat};
        \logLogSlopeTriangle{0.90}{0.3}{0.1}{3}{black};
        \addplot [mark=o, olive] table[x=MeshSize,y expr=abs(\thisrow{discN_HGradPre}-\refnormHGradPscaOne)] {outputs/pressflux/Cubic-Cells_k2_Re100_Psca1/norms.dat};
      \end{loglogaxis}            
    \end{tikzpicture}
    \subcaption{Errors on discrete $H^1(\Omega)$ norm on $p$}
  \end{minipage}
  \caption{Pressure--flux test of Section \ref{sec:pressflux}, errors on discrete norms w.r.t.~$h$}
  \label{fig:pressflux_convergence}
\end{figure}
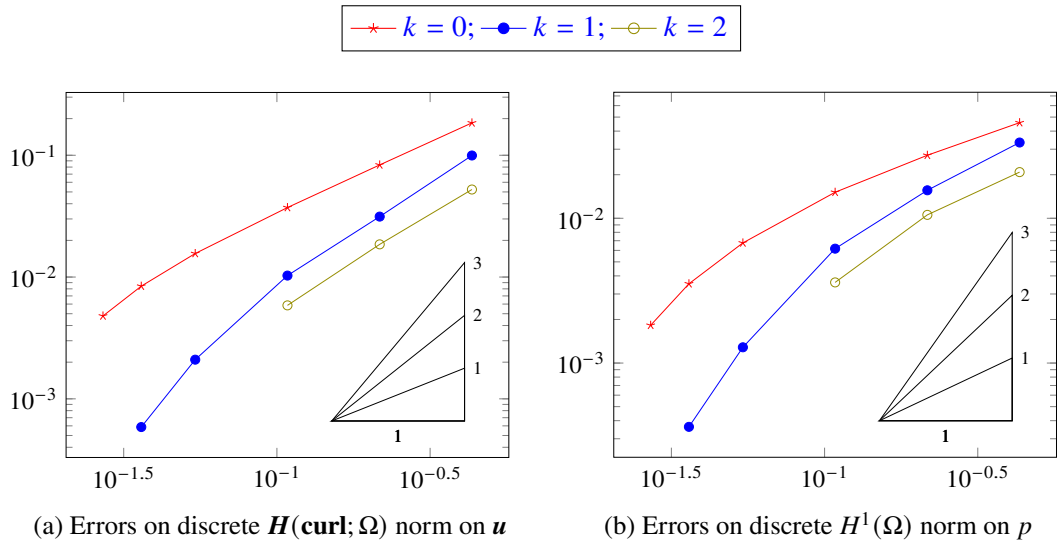


\appendix

\section{Bounds on potential reconstructions and discrete differential operators}\label{sec:bounds.pec}

We prove in this section some bounds on the discrete operators of the DDR and SDDR complex, that are used in the proof of the error estimates.

\subsection{Exterior calculus DDR complex}\label{sec:ddr.pec}

The exterior calculus DDR complex (ECDDR), introduced in \cite{Bonaldi.Di-Pietro.ea:23}, is a discretisation of the de Rham complex of differential forms on a domain $\Omega$ of any dimension $n$. When $n=3$, the spaces of 1 and 2-forms can be identified with vector fields (see \cite[Appendix A]{Bonaldi.Di-Pietro.ea:23}), and the continuous sequence becomes equivalent to the familiar vector de Rham complex. The ECDDR complex then reads
\begin{equation}\label{eq:pec.ddr.complex}
  \begin{tikzcd}
    \{0\}\arrow{r}{} & \Xsp{k}{0,h}\arrow{r}{\ued{k}{0,h}} & \Xsp{k}{1,h}\arrow{r}{\ued{k}{1,h}} &\Xsp{k}{2,h} \arrow{r}{\ued{k}{2,h}}& \Xsp{k}{3,h}\arrow{r}{} & \{0\}.
  \end{tikzcd}
\end{equation}
which, through vector proxies, can be identified with the standard vector DDR complex \cite{Di-Pietro.Droniou.ea:20,Di-Pietro.Droniou:23}
\begin{equation}\label{eq:ddr.complex}
  \begin{tikzcd}
    \{0\}\arrow{r}{}
    & \Xgrad{k}{h}\arrow{r}{\uGh{k}}
    & \Xcurl{k}{h}\arrow{r}{\uCh{k}}
    & \Xdiv{k}{h} \arrow{r}{\Dh{k}}
    & \Poly{k}(\Th)\arrow{r}{} & \{0\}.
  \end{tikzcd}
\end{equation}

The interest in adopting the exterior calculus point of view lies in the unified treatment of spaces and operators, no matter which part of the sequence they correspond to: the only difference between two related spaces/operators is the degree of the differential form in the definition, and that degree does not have any impact on the proof of certain key properties.
This is demonstrated in the rest of this section, where we establish some general results in the ECDDR framework.
Note that the results below are valid in any dimension and, when $d = 3$, they directly extend to the SDDR sequence through the bounds on the extension and reduction operators linking the DDR and SDDR sequences (see \cite{Di-Pietro.Droniou:23*1}).

We adopt the same notations as \cite[Section 3]{Bonaldi.Di-Pietro.ea:23} in the rest of this section, with the exception of swapping the placement of the polynomial degree $k$ and the degree of the differential form $l$ throughout the spaces and operators (notice that $k$ is used to indicate the degree of the differential form in \cite{Bonaldi.Di-Pietro.ea:23}).
Let us briefly describe these notations.

The set of $d$-cells of the mesh (e.g., edges if $d=1$) is $\Delta_d(\Mh)$, and the set of all cells is $\Delta(\Mh)=\bigcup_{d=0}^n\Delta_d(\Mh)$. If $f\in\Delta_d(\Mh)$ and $d'\le d$, we denote by $\Delta_{d'}(f)$ the set of $d'$-cells contained in $\overline{f}$. The discrete space of $l$-forms is
\[
 \Xsp{k}{l,h}\coloneq \bigtimes_{d=l}^n \bigtimes_{f\in \Delta_d(\Mh)}\Poly{k,-}\Lambda^{d-l}(f),
\]
where $\Poly{k,-}\Lambda^{d-l}(f)$ is the trimmed polynomial space of $(d-l)$-forms on $f$ \cite[Section 2.7]{Bonaldi.Di-Pietro.ea:23}. If $f\in\Delta_d(\Mh)$ and $d\geq l$, we denote by $\Xsp{k}{l,f}$ the restriction of $\Xsp{k}{l,h}$ to $f$, obtained by considering only the components on $f$ and its subcells.
The restriction of a vector $\underline{\omega}_h\in\Xsp{k}{l,h}$ is then $\underline{\omega}_f\in\Xsp{k}{l,f}$, and the particular polynomial component attached to $f$ is denoted by $\omega_f\in\Poly{k,-}\Lambda^{d-l}(f)$.
The interpolate of $\omega\in C^0\Lambda^l(\overline{f})$ onto $\Xsp{k}{l,f}$ is
\begin{equation}\label{eq:Iec}
  \Iec{k}{l,f}\omega\coloneq (\proj{k,-}{d'-l,f'}(\star\tr_{f'}\omega))_{f'\in\Delta_{d'}(f),\,d'\in[l,d]},
\end{equation}
where $\proj{k,-}{d'-l,f'}$ is the $L^2$-orthogonal projection on $\Poly{k,-}\Lambda^{d'-l}(f')$.

For each $f\in\Delta_d(\Mh)$ with $d\ge l$, we define the local potential reconstruction $\Pec{k}{l,f}:\Xsp{k}{l,f}\to \Poly{k}{}\Lambda^l(f)$ and, if $d\ge l+1$, a local discrete exterior derivative $\ded{k}{l,f}:\Xsp{k}{l,f}\to\Poly{k}{}\Lambda^{l+1}(f)$ recursively on the dimension $d$:
 \begin{itemize}[leftmargin=1em]
\item If $d=l$, we set 
\begin{equation}\label{eq:pec.base}
  \Pec{k}{l,f}\underline{\omega}_f \coloneq \star^{-1}\omega_f \in \Poly{k}{} \Lambda^d(f).
\end{equation}
\item
  If $l+1 \le d \le n$:
  \begin{enumerate}[leftmargin=1em]
  \item For all $\underline{\omega}_f\in\Xsp{k}{l,f}$, define $\ded{k}{l,f}\underline{\omega}_f\in\Poly{k}{}\Lambda^{l+1}(f)$ by
    \begin{multline}\label{eq:ded}
      \int_f \ded{k}{l,f}\underline{\omega}_f \wedge \mu
      = (-1)^{l+1} \int_f \star^{-1}\omega_f\wedge \ed\mu
      + \int_{\partial f} \Pec{k}{l,\partial f}\underline{\omega}_{\partial f} \wedge \tr_{\partial f}{\mu}\quad
      \forall \mu \in \Poly{k}{}\Lambda^{d-l-1}(f),
    \end{multline}
    where $\Pec{k}{l,\partial f}\underline{\omega}_{\partial f}$ is the piecewise polynomial form obtained patching $(\Pec{k}{l,f'}\underline{\omega}_{f'})_{f'\in\Delta_{d-1}(f)}$.
  \item Then, the discrete potential on the $d$-cell $f$ is given by: For all $\underline{\omega}_f\in\Xsp{k}{l,f}$,
    \begin{multline}\label{eq:Pec}
      (-1)^{l+1}\int_f \Pec{k}{l,f} \underline{\omega}_f \wedge (\ed\mu+\nu)\\
      = \int_f \ded{k}{l,f} \underline{\omega}_f \wedge \mu
      -  \int_{\partial f}  \Pec{k}{l,\partial f}\underline{\omega}_{\partial f} \wedge \tr_{\partial f}{\mu}
      + (-1)^{l+1}\int_f \star^{-1}\omega_f \wedge\nu
      \\
      \forall (\mu,\nu) \in \Koly{k+1}{d-l-1} \times \Koly{k}{d-l}(f),
    \end{multline}
    where $\Koly{a}{b}=\kappa\Poly{a-1}{}\Lambda^{b+1}(f)$ is the Koszul complement of $\ed\Poly{a+1}{}\Lambda^{b-1}(f)$ in $\Poly{a}{}\Lambda^b(f)$ ($\kappa$ is the Koszul derivative, that is, the contraction with the field $\bvec{x}-\bvec{x}_f$, where $\bvec{x}_f$ is a fixed point in $f$). 
  \end{enumerate}
\end{itemize}
The discrete derivatives in \eqref{eq:pec.ddr.complex} are obtained gathering the local discrete derivatives $\ded{k}{l,f}$ and projecting them onto the trimmed polynomial spaces of $\Xsp{k}{l+1,h}$. Through vector proxies, these definitions create the correspondences between the two diagrams \eqref{eq:pec.ddr.complex} and \eqref{eq:ddr.complex}: $l=0$ corresponds to the gradient space and discrete operator, $l=1$ to the curl space and discrete operator, and $l=2$ to the divergence space and discrete operator.

\subsection{Bounds on $\Pec{k}{l,f}$ and $\ded{k}{l,f}$}

Let $d$ be a natural number such that $0\leq d\leq n$, and $f\in\Delta_d(\mathcal{M}_h)$. For all natural numbers $l\in[0,d]$ and $s\in[1,\infty]$, define the component $L^s$-like norms on $\Xsp{k}{l,f}$ by:
\begin{equation}\label{eq:tnorm}
  \tnorm{s,l,f}{\underline{\omega}_f}\coloneq
  \sum_{f'\in\Delta_{d'}(f),d'\in[l,d]}h_{f'}^{\frac{d-d'}{s}}\norm{\Lpdf{s}{l}{f'}}{\star^{-1}\omega_{f'}}
  \qquad\forall \underline{\omega}_f\in\Xsp{k}{l,f}.
\end{equation}
Thanks to mesh regularity, this norm is equivalent uniformly in $h$ to the component norms defined in \cite[Section 4.5]{Di-Pietro.Droniou:23} when $s=2$; in passing, some of the bounds below for $s=2$ have already been established in vector proxy form in \cite[Proposition 6 and Lemma 6]{Di-Pietro.Droniou:23}.
 For future use, for any $s \in (1,\infty)$, we define the conjugate index $s'$ of $s$ such that
\begin{equation}\label{eq:conjugate index}
   \frac1s + \frac1{s'} = 1.
\end{equation}
This definition is extended to the case $s \in \{ 1, \infty \}$ setting $\frac1{\infty} \coloneq 0$ and $\frac10 \coloneq \infty$.

\begin{lemma}[Boundedness of the local discrete exterior derivative and potential]\label{lem:dP.bound}
  For any $f\in\Delta_d(\Mh)$ and $\underline{\omega}_f\in\Xsp{k}{l,f}$,
  \begin{align}\label{eq:bound.d}
    \norm{\Lpdf{s}{l+1}{f}}{\ded{k}{l,f}\underline{\omega}_f}{}&\lesssim h^{-1}_f\tnorm{s,l,f}{\underline{\omega}_f}\\\label{eq:bound.P}
    \norm{\Lpdf{s}{l}{f}}{\Pec{k}{l,f}\underline{\omega}_f}{}&\lesssim \tnorm{s,l,f}{\underline{\omega}_f}.
  \end{align}
\end{lemma}

\begin{proof}
  Since the definitions of $\ded{k}{l,f}$ and $\Pec{k}{l,f}$ are connected and recursive, we prove both bounds at the same time and by induction on $d$.

  If $d=l$, then by definition $\Pec{k}{l,f}\underline{\omega}_f=\star^{-1}\omega_f$, and the inequality \eqref{eq:bound.P} follows. \eqref{eq:bound.d} is not defined in this case.

  If $d\ge l+1$, we first establish \eqref{eq:bound.d} assuming that \eqref{eq:bound.P} holds on the boundary of $f$ (i.e. in the case $d-1$), then we prove \eqref{eq:bound.P} for $d$ using both \eqref{eq:bound.d} and \eqref{eq:bound.P} for $d-1$, which implies the general bounds via induction.

  Let $\mu\in\Pkdf{k}{d-l-1}{f}$.
  Applying H\"older inequalities to the definition \eqref{eq:ded} of $\ded{k}{l,f}$ gives
  \begin{align*}
    \int_f \ded{k}{l,f}\underline{\omega}_f\wedge\mu
    &\leq
    \norm{\Lpdf{s}{l}{f}}{\star^{-1}\omega_f}\norm{\Lpdf{s'}{d-l}{f}}{\ed\mu}+\norm{\Lpdf{s}{l}{\partial f}}{\Pec{k}{l,\partial f}\underline{\omega}_{\partial f}}\norm{\Lpdf{s'}{d-l-1}{\partial f}}{\tr_{\partial f}\mu}
    \\
    &\lesssim\norm{\Lpdf{s}{l}{f}}{\star^{-1}\omega_f}h_f^{-1}\norm{\Lpdf{s'}{d-l-1}{f}}{\mu}+\tnorm{s,l,\partial f}{\underline{\omega}_{\partial f}}h_f^{-\frac{1}{s'}}\norm{\Lpdf{s'}{d-l-1}{f}}{\mu}
  \end{align*}
  where, to pass to the second line, we have used the discrete inverse inequality \cite[Lemma 1.28]{Di-Pietro.Droniou:20} in $f$ for the first term and the boundedness \eqref{eq:bound.P} of local potentials together with the discrete trace inequality \cite[Lemma 1.32]{Di-Pietro.Droniou:20} for the second term.
  Taking the supremum over $\mu$ with $\norm{\Lpdf{s'}{d-r-1}{f}}{\mu}\leq 1$ and using $h_f \simeq h_{f'}$ (by mesh regularity) for all $f' \in \Delta_{d-1}(f)$ and the fact that $\frac{1}{s'}=1-\frac{1}{s}$ by \eqref{eq:conjugate index}, this gives
  \[
  \norm{\Lpdf{s}{l+1}{f}}{\ded{k}{r,f}\underline{\omega}_f}\lesssim h^{-1}_f\left(
  \norm{\Lpdf{s}{l}{f}}{\star^{-1}\omega_f}
  + h^{\frac{1}{s}}_{\partial f}\tnorm{s,l,\partial f}{\underline{\omega}_{\partial f}}
  \right)\lesssim h^{-1}_f\tnorm{s,l,f}{\underline{\omega}_{f}}.
  \]

  Now we show \eqref{eq:bound.P} for $l<d$. Let $(\mu,\nu)\in\Koly{k+1}{d-l-1}(f)\times\Koly{k}{d-l}(f)$. Using the discrete Lebesgue estimates of \cite[Lemma 1.25]{Di-Pietro.Droniou:20} together with the $L^2$-bounds in \cite[Lemma 9]{Di-Pietro.Droniou:23}, we have the following Sobolev--Poincar\'e inequality for $\mu$:
  \begin{equation}\label{eq:local.sobolev.poincare}
    \norm{\Lpdf{s'}{d-l-1}{f}}{\mu}\lesssim |f|^{\frac{1}{s'}-\frac{1}{2}}\norm{\Lpdf{2}{d-l-1}{f}}{\mu}
    \lesssim |f|^{\frac{1}{s'}-\frac{1}{2}}h_f\norm{\Lpdf{2}{d-l}{f}}{\ed\mu}\lesssim h_f\norm{\Lpdf{s'}{d-l}{f}}{\ed\mu}.
  \end{equation}
  Starting from the definition of $\Pec{k}{l,f}$, using the H\"older inequality to bound the integrals, then \eqref{eq:bound.d} in the first term, \eqref{eq:bound.P} and \cite[Lemma 1.32]{Di-Pietro.Droniou:20} for the boundary terms, and \eqref{eq:local.sobolev.poincare} on the norms of $\mu$ that appear in the resulting inequality, we get
  \begin{equation}\label{eq:pec.P.1}
    \begin{aligned}
      &(-1)^{l+1}\int_f \Pec{k}{l,f}\underline{\omega}_f\wedge(\ed\mu+\nu)
      \\
      &\quad\le
      \norm{\Lpdf{s}{l+1}{f}}{\ded{k}{l,f}\underline{\omega}_f}\norm{\Lpdf{s'}{d-l-1}{f}}{\mu}
      + \norm{\Lpdf{s}{l}{\partial f}}{\Pec{k}{l,\partial f}\underline{\omega}_{\partial f}}\norm{\Lpdf{s'}{d-l-1}{\partial f}}{\tr_{\partial f}\mu}
      \\
      &\qquad
      + \norm{\Lpdf{s}{l}{f}}{\star^{-1}\omega_f}\norm{\Lpdf{s'}{d-l}{f}}{\nu}
      \\
      &\quad\lesssim
      \tnorm{s,l,f}{\underline{\omega}_f}\norm{\Lpdf{s'}{d-l}{f}}{\ed\mu}+\tnorm{s,l,f}{\underline{\omega}_{\partial f}}h_f^{\frac{1}{s}}\norm{\Lpdf{s'}{d-l}{f}}{\ed\mu}\\
      &\qquad
      + \norm{\Lpdf{s}{l}{f}}{\star^{-1}\omega_f}\norm{\Lpdf{s'}{d-l}{f}}{\nu}.
    \end{aligned}
  \end{equation}
  We notice that
  \begin{align*}
    \norm{\Lpdf{s'}{d-l}{f}}{\ed\mu}+\norm{\Lpdf{s'}{d-l}{f}}{\nu}\lesssim{}& |f|^{\frac{1}{s'}-\frac12}\left(\norm{\Lpdf{2}{d-l}{f}}{\ed\mu}+\norm{\Lpdf{2}{d-l}{f}}{\nu}\right)\\
    \lesssim{}& |f|^{\frac{1}{s'}-\frac12}\norm{\Lpdf{2}{d-l}{f}}{\ed\mu+\nu}\\
    \lesssim{}&
    \norm{\Lpdf{s'}{d-l}{f}}{\ed\mu+\nu},
  \end{align*}
  where the first bound comes from the discrete Lebesgue inequality \cite[Lemma 1.25]{Di-Pietro.Droniou:20}, while the second is obtained translating \cite[Eq. (2.19)]{Di-Pietro.Droniou:23} (written in vector proxy) in the context of differential forms, and the conclusion follows applying again
  the discrete Lebesgue inequality. Plugging this estimate into \eqref{eq:pec.P.1} and taking the supremum over $(\mu,\nu)$ such that $\norm{\Lpdf{s'}{d-l}{f}}{\ed\mu+\nu}\leq 1$ yields
  \[
  \norm{\Lpdf{s}{l}{f}}{\Pec{k}{l,f}\underline{\omega}_f}
  \lesssim
  \tnorm{s,l,f}{\underline{\omega}_f}
  + h_f^{\frac{1}{s}}\tnorm{s,l,f}{\underline{\omega}_{\partial f}}
  + \norm{\Lpdf{s}{l}{f}}{\star^{-1}\omega_f}
  \lesssim\tnorm{s,l,f}{\underline{\omega}_f},
  \]
  which concludes the proof.
\end{proof}

\begin{lemma}[Discrete Lebesgue embedding]
  For all $f\in\Delta_d(\Mh)$ and all $s,t\in[1,\infty]$, we have
    \begin{equation}\label{eq:discrete.lebesgue}
      \tnorm{s,l,f}{\underline{\omega}_f}
      \simeq h_f^{d\big(\frac{1}{s}-\frac{1}{t}\big)}\tnorm{t,l,f}{\underline{\omega}_f}
      \quad\forall\underline{\omega}_f\in \Xsp{k}{l,f}.
    \end{equation}
\end{lemma}

\begin{proof}
  We only need to prove \eqref{eq:discrete.lebesgue} with $\lesssim$ instead of $\simeq$, since the converse inequality follows by reversing the roles of $s$ and $t$.
  Let $\underline{\omega}_f\in\Xsp{k}{l,f}$. For all $f'\in\Delta_{d'}(f)$ with $d'\in [l,d]$, 
  the Lebesgue embedding \cite[Lemma 1.25]{Di-Pietro.Droniou:20} together with the mesh regularity property yields
  \[
    \norm{L^s\Lambda^l(f')}{\star^{-1}\omega_{f'}}
  \lesssim h_{f'}^{d'(\frac1s-\frac1t)}
    \norm{L^t\Lambda^l(f')}{\star^{-1}\omega_{f'}}
  \le h_{f'}^{d'(\frac1s-\frac1t)}h_{f'}^{\frac{d'-d}{t}}\tnorm{t,l,f}{\underline{\omega}_f},
  \]
  where the second inequality follows from the definition \eqref{eq:tnorm} of $\tnorm{t,l,f}{\underline{\omega}_f}$.
  Write $d'{\big(\frac1s-\frac1t\big)}+\frac{d'-d}{t} = \frac{d'}{s}-\frac{d}{t}$, multiply by $h_{f'}^{\frac{d-d'}{s}}$, and sum over $f'\in\bigcup_{d'\in[l,d]}\Delta_{d'}(f)$ to obtain  
  \[
  \tnorm{s,l,f}{\underline{\omega}_f}\lesssim \left(\sum_{f'\in\Delta_{d'}(f),\,d'\in[l,d]}h_{f'}^{\frac{d}{s}-\frac{d}{t}}\right)
  \tnorm{t,l,f}{\underline{\omega}_f}.
  \]
  This concludes the proof since, by mesh regularity, $h_{f'}\simeq h_f$ and the cardinality of $\bigcup_{d'\in[l,d]}\Delta_{d'}(f)$ is bounded above uniformly in $h$ for all $f \in \Delta_d(\Mh)$.
\end{proof}

Let us define the potential-based $L^s$-norm by: for $f\in\Delta_d(\Mh)$ and $\underline{\omega}_f\in\Xsp{k}{l,f}$,
\begin{equation}\label{eq:def.pot.norm}
  \norm{s,l,f}{\underline{\omega}_f}\coloneq
  \norm{\Lpdf{s}{l}{f}}{\Pec{k}{l,f}\underline{\omega}_f}
  + \sum_{f'\in\Delta_{d'}(f),d'\in[l,d)}h_{f'}^{\frac{d-d'}{s}}\norm{\Lpdf{s}{l}{f'}}{\tr_{f'}\Pec{k}{l,f}\underline{\omega}_f - \Pec{k}{l,f'}\underline{\omega}_{f'}}.
\end{equation}
In the case $l=1$ with with vector proxies for differential forms, this norm is equivalent uniformly in $h$ to \eqref{eq:def.Ls.Xcurl}.
Additionally, for $s = 2$, the above norm is equivalent uniformly in $h$ to the one induced by the discrete $L^2$-product
\begin{equation}\label{eq:l2.prod}
  \begin{aligned}
    (\underline{\omega}_f,\underline{\mu}_f)_{l,f}
    &\coloneq
    \int_f \Pec{k}{l,f}\underline{\omega}_f \wedge \star \Pec{k}{l,f}\underline{\mu}_f
    \\
    &\quad
    + \sum_{f'\in\Delta_{d'}(f),d'\in[l,d)} h_{f'}^{d-d'} \int_{f'} \left(
      \tr_{f'}\Pec{k}{l,f}\underline{\omega}_f - \Pec{k}{l,f'}\underline{\omega}_{f'}
      \right) \wedge \star \left(
      \tr_{f'}\Pec{k}{l,f}\underline{\mu}_f - \Pec{k}{l,f'}\underline{\mu}_{f'}
      \right).
  \end{aligned}
\end{equation}

\begin{lemma}[Equivalence of potential-based and component norms]\label{lem:equiv.norms}
  For all $f \in \Delta_d(\Mh)$, the following norm equivalence holds uniformly in $h$:
  \begin{equation}\label{eq:equiv.norms}
    \norm{s,l,f}{\underline{\omega}_f}\simeq \tnorm{s,l,f}{\underline{\omega}_f}\qquad\forall \underline{\omega}_f\in\Xsp{k}{l,f}.
  \end{equation}
\end{lemma}

\begin{proof}
  To prove the first inequality $\norm{s,l,f}{\underline{\omega}_f}\lesssim \tnorm{s,l,f}{\underline{\omega}_f}$, we expand the definition of the potential-based norm using triangle inequalities, then use the discrete trace inequalities \cite[Lemma 1.32]{Di-Pietro.Droniou:20}, $h_{f'}\lesssim h_f$, and that $\tr_{f'}=\tr_{f'}\circ\cdots\circ\tr_{f}$ to write 
  \begin{equation}\label{eq:bound.trace}
    \norm{\Lpdf{s}{l}{f'}}{\tr_{f'}\Pec{k}{l,f}\underline{\omega}_f}
    \lesssim h_f^{-\frac{d-d'}{s}}\norm{\Lpdf{s}{l}{f}}{\Pec{k}{l,f}\underline{\omega}_f}.
  \end{equation}
  and finish by applying \eqref{eq:bound.P} to get
  \begin{align*}
    \norm{s,l,f}{\underline{\omega}_f}
    &\lesssim\norm{\Lpdf{s}{l}{f}}{\Pec{k}{l,f}\underline{\omega}_f}
    + \sum_{f'\in\Delta_{d'}(f),d'\in[l,d)}h_{f'}^{\frac{d-d'}{s}}\left(
      \norm{\Lpdf{s}{l}{f'}}{\tr_{f'}\Pec{k}{l,f}\underline{\omega}_f}
      + \norm{\Lpdf{s}{l}{f'}}{\Pec{k}{l,f'}\underline{\omega}_{f'}}
      \right)
      \\
      &\lesssim \sum_{f'\in\Delta_{d'}(f),d'\in[l,d]}h_{f'}^{\frac{d-d'}{s}}\tnorm{s,l,f'}{\underline{\omega}_{f'}}
      \lesssim \tnorm{s,l,f}{\underline{\omega}_f}.
  \end{align*}
where the last inequality comes from the uniform upper bound on the number of subcells of any mesh entity $f$, allowing us to absorb the outer sum into the hidden constant.

  For the reverse inequality, we first show that $\proj{k,-}{d-l,f}(\star\Pec{k}{l,f}\underline{\omega}_f)=\omega_f$. If $d=l$, the definition \eqref{eq:pec.base} of $\Pec{k}{l,f}$ yields $\proj{k,-}{0,f}(\star\Pec{k}{l,f}\underline{\omega}_f)=\omega_f$ since the trimmed space of $0$-forms is in fact the full polynomial space. Otherwise, restricting the test functions in the definition \eqref{eq:Pec} of $\Pec{k}{l,f}$ to $(\mu,\nu)\in\Koly{k}{d-l-1}(f)\times\Koly{k}{d-l}(f)\cong\Pkdf{k,-}{d-l}{f}$, we can plug in the definition \eqref{eq:ded} of $\ded{k}{l,f}$, since $\Koly{k}{d-l-1}(f)\subset\Pkdf{k,-}{d-l-1}{f}$, to obtain the desired result.

  Then re-expressing the component norm, using the isometry property of $\star$ together with the $L^s$-boundedness of the $L^2$ projections on local polynomial spaces \cite[Lemma 1.44]{Di-Pietro.Droniou:20}, taking the case $d'=d$ out of the sum, and then introducing the traces $\tr_{f'}\Pec{k}{l,f}\underline{\omega}_f$ in the remaining sum, we get
  \begin{align*}
    \tnorm{s,l,f}{\underline{\omega}_f}
    &= \sum_{f'\in\Delta_{d'}(f),d'\in[l,d]} h_{f'}^{\frac{d-d'}{s}}\norm{\Lpdf{s}{l}{f'}}{\star^{-1}\proj{k,-}{d'-l,f'}\star\Pec{k}{l,f'}\underline{\omega}_{f'}}
    \\
    &\lesssim
    \norm{\Lpdf{s}{l}{f}}{\Pec{k}{l,f}\underline{\omega}_f}
    \\
    &\quad
    + \sum_{f'\in\Delta_{d'}(f),d'\in[l,d)}h_{f'}^{\frac{d-d'}{s}}\left(
      \norm{\Lpdf{s}{l}{f'}}{\Pec{k}{l,f'}\underline{\omega}_{f'}-\tr_{f'}\Pec{k}{l,f}\underline{\omega}_f}
      + \norm{\Lpdf{s}{l}{f'}}{\tr_{f'}\Pec{k}{l,f}\underline{\omega}_f}
      \right)
      \\
      &\lesssim \norm{s,l,f}{\underline{\omega}_f},
  \end{align*}
  where again we have bounded the extra trace terms in the second line with \eqref{eq:bound.trace}.
\end{proof}

\begin{lemma}[Boundedness of the interpolator]\label{lem:I.bound} 
Let $f\in\Delta_d(\Mh)$, and $r\in\Natural$, $s\in [1,\infty]$ be such that $rs>d$. Recalling that $\Iec{k}{l,f}:C^0\Lambda^l(\overline{f})\to\Xsp{k}{l,f}$ is the interpolator on $\Xsp{k}{l,f}$, it holds
\begin{align}\label{eq:bound.I}
\tnorm{s,l,f}{\Iec{k}{l,f}\omega}\lesssim \sum_{t=0}^r h_f^t\seminorm{\Wpdf{t,s}{l}{f}}{\omega}\qquad\forall\omega\in\Wpdf{r,s}{l}{f},
\end{align}
where 
\begin{equation}\label{eq:def.seminorm}
  |\omega|_{\Wpdf{t,s}{l}{f}}\coloneq\sum_{\alpha\in\Natural^d,\,|\alpha|=t} \norm{L^s\Lambda^{l}(f)}{\partial^\alpha \omega}.
\end{equation}
\end{lemma}

\begin{remark}[Domain of the interpolator] This one-size-fits-all estimate can be improved for certain form degrees, Sobolev exponents and subcell dimension, which do not require the domain of the interpolator to be made of continuous differential forms; see, e.g., \cite[Eq.~(4.28)]{Di-Pietro.Droniou:23}.
\end{remark}

\begin{proof}
  First note that, since $rs>d$, $\omega\in \Wpdf{r,s}{l}{f}$ belongs to $C^0\Lambda^l(\overline{f})$.
  By the definition \eqref{eq:Iec} of the interpolator and \eqref{eq:tnorm} of the norm $\tnorm{s,l,f}{{\cdot}}$, 
  \begin{align*}
    \tnorm{s,l,f}{\Iec{k}{l,f}\omega}
    &=
    \sum_{f'\in\Delta_{d'}(f),d'\in[l,d]}h_{f'}^{\frac{d-d'}{s}}\norm{\Lpdf{s}{d'-l}{f'}}{\proj{k,-}{d'-l,f'}(\star\tr_{f'}\omega)}
    \\
    &\lesssim
    \sum_{f'\in\Delta_{d'}(f),d'\in[l,d]}h_{f'}^{\frac{d-d'}{s}}\norm{\Lpdf{s}{l}{f'}}{\tr_{f'}\omega}
    \\
   &\lesssim
    \sum_{f'\in\Delta_{d'}(f),d'\in[l,d]}h_{f'}^{\frac{d-d'}{s}}|f'|^{\frac{1}{s}}\sup_{\overline{f}}|\omega|
    \\
    &\lesssim
    \sum_{f'\in\Delta_{d'}(f),d'\in[l,d]}h_{f'}^{\frac{d-d'}{s}}|f'|^{\frac{1}{s}}|f|^{-\frac{1}{s}}\left(\sum_{t=0}^r h_f^t\seminorm{\Wpdf{t,s}{l}{f}}{\omega}\right),
  \end{align*}
  where for the second line we have used the $L^s$-boundedness of $L^2$-projectors on local polynomial subspaces (see \cite[Lemma 1.44]{Di-Pietro.Droniou:20}) and the fact that the Hodge star is an isometry;
  for the third, we have bounded the traces of $\omega$ by the supremum over $\overline{f}$; finally, to bound this supremum in the last line, we have invoked \cite[Eq.~(5.110)]{Di-Pietro.Droniou:20} (valid since $rs>d$). The regularity of the mesh sequence implies an upper bound on the number of subcells of dimension $d'$ of $f$ (i.e. the number of vertices, edges, and faces of any $f\in\Mh$ are bounded). Applying this, as well as the scaling $|f'|\simeq h_{f'}^{d'}$, $|f|\simeq h_{f}^{d}$ and $h_{f'}\lesssim h_{f}$, leads to \eqref{eq:bound.I}.
\end{proof}

\begin{proposition}[Boundedness of the discrete exterior derivative and potential of the interpolate]
  Let $f\in\Delta_d(\Mh)$, and $r\in\Natural$, $s\in[1,\infty]$ be such that $rs>d$. Recalling the notation \eqref{eq:def.seminorm}, it holds, for all integers $0\le m\le \min(r-1,k)$,
    \begin{equation}\label{eq:bound.dI}
      \seminorm{\Wpdf{m,s}{l+1}{f}}{\ded{k}{l,f}\Iec{k}{l,f}\omega}\lesssim \sum_{t=m+1}^rh_f^{t-m-1}\seminorm{\Wpdf{t,s}{l}{f}}{\omega}\quad
      \forall \omega\in W^{r,s}\Lambda^{l}(f)
    \end{equation}
and, for all integers $0\le m\le \min(r,k+1)$,
    \begin{equation}\label{eq:bound.PI}
      \seminorm{\Wpdf{m,s}{l}{f}}{\Pec{k}{l,f}\Iec{k}{l,f}\omega}\lesssim \sum_{t=m}^rh_f^{t-m}\seminorm{\Wpdf{t,s}{l}{f}}{\omega}\quad
      \forall \omega\in W^{r,s}\Lambda^{l}(f).
    \end{equation}
\end{proposition}

\begin{proof}
  Let us consider \eqref{eq:bound.dI}. 
  Applying \eqref{eq:bound.d} to $\Iec{k}{l,f}(\omega-\proj{k}{l,f}\omega)$ and \eqref{eq:bound.I} to $\omega-\proj{k}{l,f}\omega$ gives
  \begin{equation}\label{eq:est.bound.dI.1}
  \norm{\Lpdf{s}{l+1}{f}}{\ded{k}{l,f}\Iec{k}{l,f}(\omega-\proj{k}{l,f}\omega)}
         \lesssim h_f^{-1}\sum_{t=0}^rh_f^t \seminorm{\Wpdf{t,s}{l}{f}}{\omega-\proj{k}{l,f}\omega}.
  \end{equation}
  Since $m+1\le k+1$, the approximation properties of $\proj{k}{l,f}$ (that is, \cite[Lemma 1.45]{Di-Pietro.Droniou:20} applied to each component of the differential forms in a fixed basis) yield
  \begin{equation}\label{eq:est.bound.dI.t<=m}
    \seminorm{\Wpdf{t,s}{l}{f}}{\omega-\proj{k}{l,f}\omega}\lesssim h_f^{m+1-t}\seminorm{\Wpdf{m+1,s}{l}{f}}{\omega}\qquad\forall t\le m+1.
  \end{equation}
  On the other hand, the $W^{t,s}$-boundedness of the $L^2$-orthogonal projectors on polynomial spaces gives $\seminorm{\Wpdf{t,s}{l}{f}}{\proj{k}{l,f}\omega}\lesssim\seminorm{\Wpdf{t,s}{l}{f}}{\omega}$ (see \cite[Remark 1.47]{Di-Pietro.Droniou:20} for $t\le k+1$, the case $t>k+1$ being trivial since the left-hand side vanishes), and thus
  \begin{equation}\label{eq:est.bound.dI.t>m}
    \seminorm{\Wpdf{t,s}{l}{f}}{\omega-\proj{k}{l,f}\omega}\lesssim \seminorm{\Wpdf{t,s}{l}{f}}{\omega}\qquad\forall t\ge m+2.
  \end{equation}
  Plugging \eqref{eq:est.bound.dI.t<=m} and \eqref{eq:est.bound.dI.t>m} into \eqref{eq:est.bound.dI.1} leads to
  \[
  \norm{\Lpdf{s}{l+1}{f}}{\ded{k}{l,f}\Iec{k}{l,f}(\omega-\proj{k}{l,f}\omega)}
  \lesssim h_f^m \seminorm{\Wpdf{m+1,s}{l}{f}}{\omega} + \sum_{t=m+2}^rh_f^{t-1} \seminorm{\Wpdf{t,s}{l}{f}}{\omega}
  = \sum_{t=m+1}^rh_f^{t-1} \seminorm{\Wpdf{t,s}{l}{f}}{\omega}.
  \]
  We then use the discrete inverse inequality \cite[Lemma 1.28]{Di-Pietro.Droniou:20} on each polynomial component of $\ded{k}{l,f}\Iec{k}{l,f}(\omega-\proj{k}{l,f}\omega)$ to infer
  \[
  \seminorm{\Wpdf{m,s}{l+1}{f}}{\ded{k}{l,f}\Iec{k}{l,f}(\omega-\proj{k}{l,f}\omega)}
         \lesssim \sum_{t=m+1}^rh_f^{t-m-1} \seminorm{\Wpdf{t,s}{l}{f}}{\omega}.
  \]
 Using a triangle inequality, we next write
  \begin{align}
  \seminorm{\Wpdf{m,s}{l+1}{f}}{\ded{k}{l,f}\Iec{k}{l,f}\omega}
         \lesssim{}&  \seminorm{\Wpdf{m,s}{l+1}{f}}{\ded{k}{l,f}\Iec{k}{l,f}\proj{k}{l,f}\omega}+ \sum_{t=m+1}^rh_f^{t-m-1} \seminorm{\Wpdf{t,s}{l}{f}}{\omega}\nonumber\\
         ={}&\seminorm{\Wpdf{m,s}{l+1}{f}}{\ed\proj{k}{l,f}\omega}+ \sum_{t=m+1}^rh_f^{t-m-1} \seminorm{\Wpdf{t,s}{l}{f}}{\omega}\nonumber\\
         \lesssim {}&\seminorm{\Wpdf{m+1,s}{l}{f}}{\proj{k}{l,f}\omega}+ \sum_{t=m+1}^rh_f^{t-m-1} \seminorm{\Wpdf{t,s}{l}{f}}{\omega}\nonumber\\
         \lesssim {}&\seminorm{\Wpdf{m+1,s}{l}{f}}{\omega}+ \sum_{t=m+1}^rh_f^{t-m-1} \seminorm{\Wpdf{t,s}{l}{f}}{\omega},
         \label{eq:bound.dI.2}
  \end{align}
  where the second line follows from the polynomial consistency \cite[Eq.~(3.10)]{Bonaldi.Di-Pietro.ea:23} of $\ded{k}{l,f}$, while the conclusion is obtained by invoking the $W^{m+1,s}$-boundedness of $\proj{k}{l,f}$. The first term in the right-hand side of \eqref{eq:bound.dI.2} corresponds to the term $t=m+1$ in the sum, and this relation therefore gives \eqref{eq:bound.dI}.

\medskip

We now turn to the potential bound \eqref{eq:bound.PI}. Applying \eqref{eq:bound.P} to $\Iec{k}{l,f}(\omega-\proj{k}{l,f}\omega)$ and \eqref{eq:bound.I} to $\omega-\proj{k}{l,f}\omega$ gives
  \[
  \norm{\Lpdf{s}{l+1}{f}}{\Pec{k}{l,f}\Iec{k}{l,f}(\omega-\proj{k}{l,f}\omega)}
         \lesssim \sum_{t=0}^rh_f^t \seminorm{\Wpdf{t,s}{l}{f}}{\omega-\proj{k}{l,f}\omega}.
  \]
The conclusion is then reached as above, but writing \eqref{eq:est.bound.dI.t<=m} with $W^{m,s}$ instead of $W^{m+1,s}$ and $t\le m$ (which relaxes the condition to $m\le k+1$), and \eqref{eq:est.bound.dI.t>m} with $t\ge m+1$, and invoking the polynomial consistency \cite[Eq.\ (3.9)]{Bonaldi.Di-Pietro.ea:23} of $\Pec{k}{l,f}$.
\end{proof}

\section*{Acknowledgements}

Funded by the European Union (ERC Synergy, NEMESIS, project number 101115663).
Views and opinions expressed are however those of the authors only and do not necessarily reflect those of the European Union or the European Research Council Executive Agency. Neither the European Union nor the granting authority can be held responsible for them.
    

\printbibliography

\end{document}